\documentclass[11pt,twoside]{article}

\usepackage{jcd}
\usepackage{fullpage}
\renewcommand\eqref[1]{(\ref{#1})}

\usepackage{hyperref}

\def\tvnorm#1{\left\|{#1}\right\|_{\rm TV}}
\def\xdomain{\mc{X}}
\def\radius{R}
\def\what#1{\widehat{#1}}
\def\term{\ensuremath{S}} % Macro to allow specification of terms in equations

\def\stochmat{P} % Macro to define stochastic matrices
\def\stochvec{p} % Macro to define stochastic vectors
\def\laplacian{\ensuremath{\mc{L}}} % For graph laplacians
\def\degree{\delta} % For degrees of vertices in graphs
\newcommand{\vol}{\mathop{\rm vol}} % Volume of graph
\def\stepsize{\alpha} % Stepsize scalar
\def\failprob{\rho} % Probability of failure
\renewcommand\P{\mathbb{P}} % Probability symbol
\def\defeq{\triangleq} % How to do definitions
\def\ninv{\frac{1}{n}}
\def\zb{\ensuremath{\bar{z}}}
\def\onevec{\ensuremath{1\!\!1}}
\def\compositeproject#1{\Pi_\xdomain^{#1}}
\def\norm#1{\ensuremath{\left\|#1\right\|}}
\def\dnorm#1{\ensuremath{\left\|#1\right\|_{*}}}
\def\dnormb#1{\ensuremath{\bigg\|{#1}\bigg\|_*}}
\def\prox{\psi}
\def\proxdual{\psi^*}
\def\hinge#1{\left[{#1}\right]_+}
\def\bindic#1{\mathbb{I}\left({#1}\right)}
\def\neighbor{N}
\newcommand{\order}{\ensuremath{\mathcal{O}}}
\renewcommand{\defeq}{\ensuremath{\; : \, = }}
\newcommand{\numnode}{\ensuremath{n}}
\newcommand{\Var}{\ensuremath{\mathop{\rm Var}}}

\newcommand{\mydegree}{\ensuremath{\delta}}
\newcommand{\graph}{\ensuremath{G}}
\newcommand{\vertex}{\ensuremath{V}}

\newcommand{\dmax}{\ensuremath{\mydegree_{\operatorname{max}}}}
\newcommand{\dmin}{\ensuremath{\mydegree_{\operatorname{min}}}}

\newcommand{\Tspec}[2]{\ensuremath{T_\graph(#1; #2)}}

\newcommand{\real}{\ensuremath{\R}}

\newcommand{\Fmat}{\ensuremath{F}}

\newcommand{\Sstar}{\ensuremath{S}}
 
\newcommand{\NUMSIM}{20}

\makeatletter
\long\def\@makecaption#1#2{
        \vskip 0.8ex
        \setbox\@tempboxa\hbox{\small {\bf #1:} #2}
        \parindent 1.5em  %% How can we use the global value of this???
        \dimen0=\hsize
        \advance\dimen0 by -3em
        \ifdim \wd\@tempboxa >\dimen0
                \hbox to \hsize{
                        \parindent 0em
                        \hfil 
                        \parbox{\dimen0}{\def\baselinestretch{0.96}\small
                                {\bf #1.} #2
                                } 
                        \hfil}
        \else \hbox to \hsize{\hfil \box\@tempboxa \hfil}
        \fi
        }
\makeatother

\begin{document}

\begin{center}
  {\bf{\Large{
        Dual Averaging for Distributed Optimization: \\
        Convergence Analysis and Network Scaling}}} \\

\vspace*{.2in}

\begin{tabular}{ccc}
 John C. Duchi$^1$ & Alekh Agarwal$^1$ &  Martin J. Wainwright$^{1,2}$ \\
{\texttt{jduchi@cs.berkeley.edu}} &
{\texttt{alekh@cs.berkeley.edu}} &
{\texttt{wainwrig@eecs.berkeley.edu}}
\end{tabular}

\vspace*{.2in}

\begin{tabular}{ccc}
Department of Electrical Engineering and Computer Sciences$^1$ & &
Department of Statistics$^2$   \\
UC Berkeley, Berkeley, CA 94720 & &
UC Berkeley, Berkeley, CA 94720 
\end{tabular}

\vspace*{.2in}

\today

\begin{abstract}
The goal of decentralized optimization over a network is to optimize a
global objective formed by a sum of local (possibly nonsmooth) convex
functions using only local computation and communication.  It arises
in various application domains, including distributed tracking and
localization, multi-agent co-ordination, estimation in sensor
networks, and large-scale optimization in machine learning.  We
develop and analyze distributed algorithms based on dual averaging of
subgradients, and we provide sharp bounds on their convergence rates as a
function of the network size and topology. Our method of analysis
allows for a clear separation between the convergence of the
optimization algorithm itself and the effects of communication
constraints arising from the network structure.  In particular, we
show that the number of iterations required by our algorithm scales
inversely in the spectral gap of the network.  The sharpness of this
prediction is confirmed both by theoretical lower bounds and
simulations for various networks.  Our approach includes both the
cases of deterministic optimization and communication, as well as
problems with stochastic optimization and/or communication.
\end{abstract}
\end{center}

%%%%%%%%%%%%%%%%%%%%%%%%%%%%%%%%%%%%%%%%%%%%%%%%%%%%%%%%%%%%%%%%%%%%%%%%%%%%%%%

\section{Introduction}
\label{sec:intro}

The focus of this paper is the development and analysis of distributed
algorithms for solving convex optimization problems that are defined
over networks.  Such network-structured optimization problems arise in
a variety of application domains within the information sciences and
engineering.  For instance, problems such as multi-agent coordination,
distributed tracking and localization, estimation problems in sensor
networks and packet routing are all naturally cast as distributed
convex
minimization~\cite{BertsekasTs89,LiWoHuSa02,LesserOrTa03,RabbatNo04,XiaoBoKi07}.
Common to these problems is the necessity for completely decentralized
computation that is locally light---so as to avoid overburdening small
sensors or flooding busy networks---and robust to periodic link or
node failures. As a second example, data sets that are too large to be
processed quickly by any single processor present related
challenges. A canonical example that arises in statistical machine
learning is the problem of minimizing a loss function averaged over a
large dataset (e.g., optimization in support vector
machines~\cite{CortesVa95}). With terabytes of data, it is desirable
to assign smaller subsets of the data to different processors, and the
processors must communicate to find parameters that minimize the loss
over the entire dataset. However, the communication should be
efficient enough that network latencies do not offset computational
gains.

Distributed computation has a long history in optimization.  Primal
and dual decomposition methods that lend themselves naturally to a
distributed paradigm have been known for at least fifty years, and
their behavior is well understood
(e.g.,~\cite{DantzigWo60,Bertsekas99}).  The 1980s saw significant
interest in distributed detection, consensus, and minimization. The
seminal work of Tsitsiklis and
colleagues~\cite{Tsitsiklis84,TsitsiklisBeAt86,BertsekasTs89} analyzed
algorithms for minimization of a smooth function $f$ known to several
agents while distributing processing of components of the parameter
vector $x \in \R^n$. An important special case of network optimization---with
much faster convergence rates than those known for general distributed
optimization---is consensus averaging, where each processor in the
network must agree on a single (vector-valued) variable. This is
recovered from our objective~\eqref{eqn:objective} by setting $f_i(x)
= \|x - \theta_i\|_2^2$. A number of researchers have obtained sharp
convergence results for distributed consensus algorithms by studying
network topology and using spectral properties of random walks or path
averaging arguments on the underlying graph structure (e.g.,
see~\cite{BoydGhPrSh06,BenezitDiThVe08,DimakisSaWa08} and references
therein). 
Allowing stochastic gradients also lets us tackle distributed averaging with
noise~\cite{XiaoBoKi07}. Mosk-Aoyama et al.~\cite{Mosk-AoyamaRoSh2010}
consider a problem related to our setup, minimizing $\sum_{i=1}^n f_i(x_i)$
for $x_i \in \R$ subject to linear equality constraints, and they obtain rates
of convergence dependent on network-conductance using an algorithm similar to
dual decomposition.
More recently, a few
researchers have shifted focus to problems in which each processor
locally has its own convex (potentially non-differentiable) objective
function~\cite{NedicOz09,LobelOz09}. Whereas these initial papers
treated the case of unconstrained optimization, more recent work by
Ram et al.~\cite{RamNeVe10} analyzes a projected subgradient algorithm
for distributed minimization of non-smooth functions with constraints.

Our paper makes two main contributions.  The first contribution is to
provide a new simple subgradient algorithm for distributed constrained
optimization of a convex function; we refer to it as a \emph{dual
averaging subgradient method}, since it is based on maintaining and
forming weighted averages of subgradients throughout the network.
This approach is essentially different from previously developed
methods~\cite{NedicOz09,LobelOz09,RamNeVe10}, and these differences
facilitate our analysis of network scaling issues, meaning how
convergence rates depend on network size and topology.  Indeed, the
second main contribution of this paper is a careful analysis that
demonstrates a close link between convergence of the algorithm and the
underlying spectral properties of the network. Our analysis splits the
convergence rate of the algorithm into two terms: an optimization term
and a network deviation term. We obtain the optimization penalty using
techniques based on the optimization literature, specifically building
on results due to Nesterov~\cite{Nesterov09}.  This splitting approach
can also be adapted to naturally handle issues such as constrained
optimization, stochastic communication, and stochastic optimization
due to elegant properties of the dual averaging algorithm. On the
other hand, the network scaling terms are obtained using techniques
from analysis of Markov chains coupled with the distributed
communication protocol.  We show that the network deviation terms we
derive are sharp for our algorithm; in the special case of the
consensus problem, these terms are known to be
near-optimal~\cite{BoydGhPrSh06}.

By comparison to previous work, our convergence results and proofs are
different, and our characterization of the network scaling terms are
often much stronger.  In particular, the convergence rates given by
the papers~\cite{NedicOz09,LobelOz09} grow exponentially in the number
of nodes $n$ in the network. Nedi\'{c} et al.~\cite{NedicOlOzTs09} and
Ram et al.~\cite{RamNeVe10} provide a much tighter analysis that
yields convergence rates that scale polynomially in the network size,
but are independent of the network topology (apart from requiring
strong connectedness and degree independent of $n$).  Specifically,
Corollary 5.5 in the paper~\cite{RamNeVe10} guarantees that their
projected subgradient algorithm---under the assumptions that the
number of time steps is known a priori and the stepsize is chosen
optimally---obtains an $\epsilon$-optimal solution to the optimization
problem in $\order(n^3 / \epsilon^2)$ time.  Since this bound is
essentially independent of network topology, it does not capture the
intuition that distributed algorithms should converge much faster on
``well-connected'' networks---expander graphs being a prime
example---than on poorly connected networks (e.g., chains, trees or
single cycles). Johansson et al.~\cite{JohanssonRaJo09} analyze a low
communication peer-to-peer protocol that attains rates dependent on
network structure. However, in their algorithm only one node has a
current parameter value, while all nodes in our algorithm maintain
good estimates of the optimum at all time.  This is important in
online, streaming, and control problems where nodes are expected to
act or answer queries in real time.  In additional comparison to
previous work, our analysis gives network scaling terms that are often
substantially sharper. Our development yields an algorithm with
convergence rate that scales inversely in the spectral gap of the
network.  By exploiting known results on spectral gaps for graphs with
$n$ nodes, we show that (disregarding logarithmic factors) our
algorithm obtains an $\epsilon$-optimal solution in $\order(n^2 /
\epsilon^2)$ iterations for a single cycle or path,
$\order(n/\epsilon^2)$ iterations for a two-dimensional grid, and
$\order(1/\epsilon^2)$ iterations for a bounded degree expander graph.
Moreover, simulation results show an excellent agreement with these
theoretical predictions.

Our analysis covers several settings for distributed minimization. We
begin by studying fixed communication protocols, which are of interest
in a variety of areas such as cluster computing or sensor networks
with a fixed hardware-dependent protocol. We also investigate
randomized communication protocols as well as randomized network
failures, which are often essential to handle gracefully in wireless
sensor networks and large clusters with potential node failures.
Randomized communication also provides an interesting tradeoff between
communication savings and convergence rates. In this setting, we
obtain much sharper results than previous work by studying the
spectral properties of the expected transition matrix of a random walk
on the underlying graph. We also present an analysis of our algorithm
with stochastic gradient information, which is not difficult when
combined with our initial theorems. We describe an extension for
(structured) regularized objectives that often arise in machine
learning problems in Appendix~\ref{app:composite}.

The remainder of this paper is organized as follows.  Section~\ref{sec:setup}
is devoted to a formal statement of the problem and description of the dual
averaging algorithm, whereas Section~\ref{sec:results} states the main results
and consequences of our paper. Having stated our main results, in
Section~\ref{sec:related-work} we give a more in-depth comparison of our work
with other recent work. In Section~\ref{sec:master-convergence}, we state and
prove basic convergence results on the dual averaging algorithm, which we then
exploit in Section~\ref{sec:fixed-p-convergence} to derive concrete results
that depend on the spectral gap of the network.
Sections~\ref{sec:random-P-convergence} and~\ref{sec:stoch-gradient} treat
extensions with noise, in particular algorithms with noisy communication and
stochastic gradient methods respectively.  In Section~\ref{sec:simulations},
we present the results of simulations that confirm the sharpness of our
analysis.

\vspace*{.2cm}

\noindent {\bf{Notation:}} We collect some notation used throughout
the paper.  We use $\onevec$ to denote the all-ones vector. We also
use standard asymptotic notation for sequences.  If $a_n$ and $b_n$
are positive sequences, then $a_n = \order(b_n)$ means that $\limsup_n
a_n / b_n < \infty$, whereas $a_n = \Omega(b_n)$ means that $\liminf_n
a_n / b_n > 0$.  On the other hand, $a_n = o(b_n)$ means that $\lim_n
a_n / b_n = 0$ and $a_n = \omega(b_n)$ means that $\lim_n a_n/b_n =
\infty$. Finally, we write $a_n = \Theta(b_n)$ if $a_n = \order(b_n)$
and $a_n = \Omega(b_n)$.

\section{Problem set-up and algorithm}
\label{sec:setup}

In this section, we provide a formal statement of the distributed
minimization problem, and a description of the distributed dual
averaging algorithm.

\subsection{Distributed minimization}

We consider an optimization problem based on functions that are
distributed over a network.  More specifically, let $G = (V, E)$ be an
undirected graph over the vertex set $V = \{1, 2, \ldots, \numnode\}$ with
edge set $E \subset V \times V$.  Associated with each $i \in V$ is
convex function $f_i:\R^d \rightarrow \R$, and our overarching goal is
to minimize the sum
\begin{equation*}
f(x)  = \ninv\sum_{i=1}^n f_i(x),
\end{equation*}
subject to the constraint that $x \in \R^d$ belongs to some closed convex set
$\xdomain$---that is, solve the problem
\begin{equation}
  \label{eqn:objective}
  \min_x ~ \ninv\sum_{i=1}^n f_i(x) \qquad \subjectto ~ x \in
  \xdomain.
\end{equation}
Each function $f_i$ is convex and hence sub-differentiable, but need
not be smooth. We assume without loss of generality that $0 \in
\xdomain$, since we can simply translate $\xdomain$. Each node $i \in
V$ is associated with a separate agent, and each agent $i$ maintains
its own parameter vector $x_i \in \R^d$.  The graph $G$ imposes
communication constraints on the agents: in particular, agent $i$ has
local access to only the objective function $f_i$ and can communicate
directly only with its immediate neighbors $j \in \neighbor(i) \defeq
\{j \in V \, \mid \, (i,j) \in E \}$.

Problems of this nature arise in a variety of application domains, and
as motivation for the analysis to follow, let us consider a few
here. A first example is a sensor network, in which each agent
represents a sensor mote, equipped with a radio transmitter for
communication, some basic sensing devices, and some local memory and
computational power.  In environmental applications of sensor networks,
each mote $i$ might take a measurement $y_i$ of the temperature, and
the global objective could be to compute the median of the
measurements $\{y_1, y_2, \ldots, y_n \}$.  This median computation
problem can be formulated as minimizing the scalar objective function
$\frac{1}{n} \sum_{i=1}^n f_i(x)$, where $f_i(x) = |x - y_i|$.
Similar formulations apply to the problem of computing other statistics
such as means, variances, quantiles and other $M$-estimators.

A second motivating example is the machine learning problem first described in
Section~\ref{sec:intro}.  In this case, the set $\xdomain$ is the parameter
space of the statistician or learner. Each function $f_i$ is the empirical
loss over the subset of data assigned to the $i$th processor, and assuming
that each subset is of equal size (or that the $f_i$ are normalized suitably),
the average $f$ is the empirical loss over the entire dataset.  Here we use
cluster computing as our computational model, where each processor is a node
in the cluster, and the graph $G$ contains edges between those
processors that are directly connected with small network latencies. A special
case of our optimization problem within this computational model is the
distributed perceptron, recently considered by McDonald et
al.~\cite{McDonaldHaMa2010}.

\subsection{Standard dual averaging}

Our algorithm is based on a projected dual averaging
algorithm~\cite{Nesterov09}, designed for minimization of a
(potentially nonsmooth) convex function $f$ subject to the constraint
$x \in \xdomain$.  We begin by describing the standard version of this
algorithm, and then discuss the extensions for the distributed setting
of interest in this paper.

The dual averaging scheme is based on a \emph{proximal function}
$\prox: \R^d \rightarrow \R$ that is assumed to be $1$-strongly convex
with respect to some norm $\norm{\cdot}$---more precisely, the
proximal function satisfies
\begin{equation*}
\qquad
\prox(y) \geq \prox(x) + \<\nabla \prox(x), y - x\> + \half
\norm{x - y}^2 \qquad \mbox{for all $x, y \in \xdomain$.}
\end{equation*}
In addition, we assume that $\prox(x) \geq 0$ for all $x \in \xdomain$
and that $\prox(0) = 0$; these are standard assumptions that can be
made without loss of generality. Examples of such proximal function
and norm pairs include:
\begin{itemize}
\item the quadratic $\prox(x) = \half \ltwo{x}^2$, which is the
  canonical proximal function. Clearly $\half\ltwo{0}^2 = 0$, and
  $\half \ltwo{x}^2$ is strongly convex with respect to the
  $\ell_2$-norm for $x \in \R^d$.
\item the entropic function \mbox{$\prox(x) = \sum_{i=1}^d x_i \log
  x_i - x_i$}, which is strongly convex with respect to the
  $\ell_1$-norm for $x$ in the probability simplex, $\{x \mid x
  \succeq 0, \<x, \onevec\> = 1\}$.
\end{itemize}

We assume that each function $f_i$ is \emph{$L$-Lipschitz} with
respect to the same norm $\norm{\cdot}$---that is,
\begin{align}
  \label{eqn:lipschitz}
  |f_i(x) - f_i(y)| & \leq L \norm{x - y} \qquad \mbox{for $x, y \in
    \xdomain$.}
\end{align}
There are many cost functions $f_i$ that satisfy this type of
Lipschitz condition.  For instance, it holds for any convex function
on a compact domain $\xdomain$, or for any polyhedral function on an
arbitrary domain~\cite{HiriartUrrutyLe96}. Note that the Lipschitz
condition~\eqref{eqn:lipschitz} implies that for any $x \in \xdomain$
and any subgradient $g_i \in \partial f_i(x)$, we have
\begin{equation*}
  \dnorm{g_i} \leq L,
\end{equation*}
where $\dnorm{\cdot}$ denotes the \emph{dual norm} to $\norm{\cdot}$, defined
by $\dnorm{v} \defeq \sup_{\norm{u} = 1} \<v, u\>$.

The dual averaging algorithm generates a sequence of iterates
$\{x(t), z(t) \}_{t=0}^\infty$ contained within $\xdomain \times \R^d$
according to the following steps.  At time step $t$ of the algorithm,
it receives a subgradient $g(t) \in \partial f(x(t))$, and then
performs the updates
\begin{equation}
  z(t + 1) = z(t) + g(t) ~~ \quad \mbox{ and } \quad ~~~~ x(t + 1) =
  \Pi_{\xdomain}^\prox(z(t + 1), \stepsize(t)),
  \label{eqn:standard-lazy}
\end{equation}
where $\{\stepsize(t)\}_{t=0}^\infty$ is a non-increasing sequence of
positive stepsizes and
\begin{equation}
  \Pi_\xdomain^\prox(z, \alpha)
  \defeq \argmin_{x \in \xdomain} \biggr\{\<z, x\> + \frac{1}{\alpha}
  \prox(x)\biggr\}
  \label{eqn:projection}
\end{equation}
is a type of projection.  The intuition underlying this algorithm is
as follows: given the current iterate $(x(t), z(t))$, the next iterate
$x(t + 1)$ to chosen to minimize an averaged first-order approximation
to the function $f$, while the proximal function $\prox$ and stepsize
$\stepsize(t) > 0$ enforce that the iterates $\{x(t)\}_{t=0}^\infty$
do not oscillate wildly. The algorithm is similar to the follow the
perturbed leader and lazy projection algorithms developed in the
context of online optimization~\cite{KalaiVe05}, though in this form
the algorithm seems to be originally due to
Nesterov~\cite{Nesterov09}.
In Section~\ref{sec:master-convergence}, we show that a simple
analysis of the convergence of the above procedure allows us to relate
it to the distributed algorithm we describe.

\subsection{Distributed dual averaging}

We now consider an appropriate and novel extension of dual averaging
to the distributed setting.  At each iteration $t = 1, 2, 3, \ldots$,
the algorithm maintains $n$ pairs of vectors $(x_i(t), z_i(t)) \in
\xdomain \times \R^d$, with the $i^{th}$ pair associated with node $i
\in V$.  At iteration $t$, each node $i \in V$ computes an element
$g_i(t) \in \partial f_i(x_i(t))$ in the subdifferential of the local
function $f_i$ and receives information about the parameters
$\{z_j(t), \; j \in \neighbor(i) \}$ associated with nodes $j$ in its
neighborhood $\neighbor(i)$.  Its update of the current estimated
solution $x_i(t)$ is based on a convex combination of these
parameters.  To model this weighting process, let $\stochmat \in R^{n
\times n}$ be a matrix of non-negative weights that respects
the structure of the graph $G$, meaning that for $i \ne j$,
$\stochmat_{ij} > 0$ only if $(i,j) \in E$.  We assume that
$\stochmat$ is a doubly stochastic matrix, so that
\begin{equation*}
  \sum_{j=1}^n \stochmat_{ij} \; = \; \sum_{j \in \neighbor(i)} \stochmat_{ij}
  \; = \; 1 \qquad {\rm for~all~} i \in V,
  \quad\sum_{i=1}^n\stochmat_{ij}
  = \sum_{i \in \neighbor(j)}\stochmat_{ij} = 1 \qquad
  {\rm for~all}~ j \in V.
\end{equation*}

Using this notation, given the non-increasing sequence $\{\stepsize(t)
\}_{t=0}^\infty$ of positive stepsizes, each node $i \in V = \{1, 2,
\ldots, n\}$ performs the updates
\begin{subequations}
\begin{align}
  \label{eqn:lazy-unproject} 
  z_i(t + 1) & = \sum_{j \in \neighbor(i)} \stochvec_{ij} z_j(t) + g_i(t),
  \quad \mbox{and} \\
  x_i(t + 1) & = \Pi_{\xdomain}^\prox(z_i(t + 1), \stepsize(t)),
  \label{eqn:lazy-project}
\end{align}
\end{subequations}
where the projection $\Pi_{\xdomain}^\prox$ was defined
previously~\eqref{eqn:projection}.  In words, node $i$ computes the
new dual parameter $z_i(t+1)$ from a weighted average of its own
subgradient $g_i(t)$ and the parameters $\{z_j(t), j \in \neighbor(i)
\}$ in its own neighborhood $\neighbor(i)$, and then computes the next
local iterate $x_i(t+1)$ by a projection defined by the proximal
function $\prox$ and stepsize $\stepsize(t) > 0$.

In the sequel, we show convergence of the local sequence
$\{x_i(t)\}_{t=1}^\infty$ to the optimum of \eqref{eqn:objective} via
the \emph{running local average}
\begin{equation}
  \what{x}_i(T) = \frac{1}{T} \sum_{t=1}^T x_i(t).
  \label{eqn:local-average-def}
\end{equation}
Note that this quantity is locally defined at node $i$ and so can be
computed in a distributed manner.  From the definition of updates, it
is clear that each element of the sequence $\{z_i(t)\}_{t=0}^\infty$
is essentially a weighted average of the gradients seen so far, which
is a natural extension of dual averaging. At the same time, as we
shall see, the averaging of the dual parameters in the sequence
$\{z_i(t)\}_{t=0}^\infty$ allows us to neatly sidestep the complexity
arising from non-linearity of projections. We will thus be able to generalize
the algorithm from equations~\eqref{eqn:lazy-unproject}
and~\eqref{eqn:lazy-project} to the case where $\stochmat$ is random
and varies with time as well as when the vectors $g_i(t)$ are noisy
versions of subgradients, satisfying only $\E[g_i(t)] \in \partial
f_i(x_i(t))$.

\section{Main results and consequences}
\label{sec:results}

We will now state the main results of this paper and illustrate some
of their consequences. We give the proofs and a deeper investigation
of related corollaries at length in the sections that follow.

\subsection{Convergence of distributed dual averaging}

We start with a result on the convergence of the distributed dual
averaging algorithm that provides a decomposition of the error into an
optimization term and the cost associated with network communication.
In order to state this theorem, we define the averaged dual variable
\mbox{$\bar{z}(t) \defeq \frac{1}{n} \sum_{i=1}^n z_i(t)$,} and we
recall the definition~\eqref{eqn:local-average-def} of the local
average $\what{x}_i(T)$.

\begin{theorem}[Basic convergence result]
  \label{theorem:master-convergence}
  Let the sequences $\{x_i(t)\}_{t=0}^\infty$ and
  $\{z_i(t)\}_{t=0}^\infty$ be generated by the
  updates~\eqref{eqn:lazy-unproject} and~\eqref{eqn:lazy-project} with
  step size sequence $\{\stepsize(t)\}_{t=0}^\infty$. Then
  for any $x^* \in \xdomain$ and for each
  node $i \in V$, we have
  \begin{align}
  \label{eqn:master-convergence}
    f(\what{x}_i(T)) - f(x^*)
    & \le \frac{1}{T \stepsize(T)}\prox(x^*) + \frac{L^2}{2T}
    \sum_{t=1}^T\stepsize(t-1) \nonumber \\
    & ~~~ + \frac{2 L}{nT} \sum_{t=1}^T \sum_{j=1}^n
    \stepsize(t) \dnorm{\bar{z}(t) - z_j(t)}
    + \frac{L}{T} \sum_{t=1}^T
    \stepsize(t)\dnorm{\bar{z}(t) - z_i(t)}.
  \end{align}
\end{theorem}

% Distracting and not needed here
%For any
%  $x^* \in \xdomain$ and $i \in V$, we have
%  \begin{align}
%    \sum_{t=1}^T f(x_i(t)) - f(x^*)
%    & \le \frac{1}{\stepsize(T)} \prox(x^*) + \frac{L^2}{2}
%    \sum_{t=1}^T\stepsize(t-1) \nonumber \\
%    & ~~~ + \frac{2 L}{n}
%    \sum_{t=1}^T \sum_{j=1}^n \stepsize(t)\dnorm{\bar{z}(t) - z_j(t)}
%    + L\sum_{t=1}^T \stepsize(t)\dnorm{\bar{z}(t) - z_i(t)}.
%    \label{eqn:master-convergence}
%  \end{align}

Theorem~\ref{theorem:master-convergence} guarantees that after $T$
steps of the algorithm, every node $i \in V$ has access to a locally
defined quantity $\what{x}_i(T)$ such that the difference
$f(\what{x}_i(T)) - f(x^*)$ is upper bounded by a sum of four terms.
The first two terms in the upper bound~\eqref{eqn:master-convergence}
are optimization error terms that are common to subgradient
algorithms. The third and fourth terms are penalties incurred due to
having different estimates at different nodes in the network, and they
measure the deviation of each node's estimate of the average gradient
from the true average gradient.\footnote{The fact that the term
  $\dnorm{\bar{z}(t) - z_i(t)}$ appears an extra time is no
  significant difficulty, as we will bound the difference $\bar{z}(t)
  - z_i(t)$ uniformly for all $i$ when giving concrete convergence
  results.} Thus, roughly, Theorem~\ref{theorem:master-convergence}
ensures that as long the bound on the deviation $\dnorm{\bar{z}(t) -
  z_i(t)}$ is tight enough, for appropriately chosen $\stepsize(t)$
(say $\stepsize(t) \approx 1/\sqrt{t}$), the error of $\what{x}_i(T)$
is small uniformly across all nodes $i \in V$, and asymptotically
approaches 0. See Theorem~\ref{theorem:simple-convergence} in the next
section for a precise statement of rates.

It is worthwhile comparing the optimization error term from the
bound~\eqref{eqn:master-convergence} to known results. Subgradient
descent on the average function $f = \frac{1}{n} \sum_{i=1}^n f_i$ has
identical convergence rate, as does the randomized version of
incremental subgradient descent~\cite{NedicBe01}. However, the
distributed nature of the algorithm gives a computational advantage
over full gradient descent---the gradient computation requires
$\order(1)$ computation per computer rather than $\order(n)$ on a
single computer. To highlight the benefits compared to incremental
subgradient descent, consider the common problem in machine learning
and statistics of minimizing a loss on a large dataset. A randomized
incremental gradient descent method must access random subsets of data
at every iteration, leading to randomized disk seeks with high
latency, which the distributed algorithm avoids. In addition, we
expect (and empirically see that this is indeed the case) our method
to produce more stable iterates, as we observe the gradients of all
$n$ functions at every round, albeit with a network communication lag.

\subsection{Convergence rates and network topology}

We now turn to investigation of the effects of network topology on
convergence rates.  In this section, we assume that the network topology is
static and that communication occurs via a fixed doubly stochastic
weight matrix $\stochmat$ at every round.\footnote{In later sections, we
weaken these conditions.}  Since $\stochmat$ is
doubly stochastic, it has largest singular value
$\sigma_1(\stochmat) = 1$.  As summarized in the following result, the
convergence rate of the distributed projection algorithm is controlled
by the \emph{spectral gap} $\gamma(\stochmat) \defeq
1-\sigma_2(\stochmat)$ of the matrix $\stochmat$.
\begin{theorem}[Rates based on spectral gap]
  \label{theorem:simple-convergence}
  Under the conditions and notation of
  Theorem~\ref{theorem:master-convergence}, suppose moreover that
  $\prox(x^*) \le R^2$.  With step size choice \mbox{$\stepsize(t)
    = \frac{R\sqrt{1 - \sigma_2(\stochmat)}}{4L \sqrt{t}}$,} we
  have
  \begin{equation*}
    f(\what{x}_i(T)) - f(x^*)
    \leq 8 \frac{R L}{\sqrt{T}} 
    \frac{\log(T \sqrt{n})}{\sqrt{1 - \sigma_2(\stochmat)}} \qquad
    \mbox{for all $i \in \vertex$.}
  \end{equation*}
\end{theorem}

To the best of our knowledge, this theorem is the first to establish a
tight connection between the convergence rate of distributed
subgradient methods to the spectral properties of the underlying
network.  In particular, the inverse dependence on the spectral gap
$1-\sigma_2(\stochmat)$ is quite natural, since it is well-known to
determine the rates of mixing in random walks on
graphs~\cite{LevinPeWi08}, and the propagation of information in our
algorithm is integrally tied to the random walk on the underlying
graph with transition probabilities specified by $\stochmat$. \\

Using Theorem~\ref{theorem:simple-convergence}, one can derive
explicit convergence rates for several classes of interesting
networks, and Figure~\ref{fig:graph-types} illustrates four different
graph topologies that are of interest.  As a first example, the
$k$-connected cycle in panel (a) is formed by placing $n$ nodes on a
circle and connecting each node to its $k$ neighbors on the right and
left.  For small $k$, the cycle graph is rather poorly connected, and
our analysis will show that this leads to slower convergence rates
than other graphs with better connectivity.  The grid graph in two
dimensions is obtained by connecting nodes to their $k$
nearest neighbors in axis-aligned directions.  For instance, panel (b)
shows an example of a degree $4$ grid graph in two-dimensions.  Both
the cycle and grid topologies are possible models for clustered
computing as well as sensor networks.

In panel (c), we show a random geometric graph, constructed by placing
nodes uniformly at random in $[0, 1]^2$ and connecting any two nodes
separated by a distance less than some radius $r > 0$.  These graphs
are used to model the connectivity patterns of devices, such as
wireless sensor motes, that can communicate with all nodes in some
fixed radius ball, and have been studied extensively
(e.g.,~\cite{GuptaKu00,Penrose03}).  There are natural generalizations
to dimensions $d > 2$ as well as to cases in which the spatial
positions are drawn from a non-uniform distribution.

Finally, panel (d) shows an instance of a bounded degree expander,
which belongs to a special class of sparse graphs that have very good
mixing properties~\cite{Chung98}.  Expanders are an attractive option
for the network topology in distributed computation since they are
known to have large spectral gaps.  For many random graph models, a
typical sample is an expander with high probability; for instance, a
randomly chosen bipartite graph satisfies this property~\cite{Alon86},
as do random degree regular graphs~\cite{FriedmanKaSz89}.  In
addition, there are several deterministic constructions of expanders
that are degree regular (see Section 6.3 of Chung~\cite{Chung98} for
further details).  The deterministic constructions are of interest
because they can be used to design a network, while the random
constructions are of interest since they are often much simpler. \\

In order to state explicit convergence rates, we need to specify a
particular choice of the matrix $\stochmat$ that respects the graph
structure.  Although many such choices are possible, here we focus on
the graph Laplacian~\cite{Chung98}.  First, we let $A \in \real^{n
  \times n}$ be the symmetric adjacency matrix of the undirected graph
$G$, satisfying $A_{ij} = 1$ when $(i,j) \in E$ and $A_{ij} = 0$
otherwise.  For each node $i \in V$, we let $\mydegree_i = |N(i)| =
\sum_{j=1}^n A_{ij}$ denote the degree of node $i$, and we define the
diagonal matrix $D = \diag \{\mydegree_1, \ldots, \mydegree_n\}$. We
assume that the graph is connected, so that $\mydegree_i \geq 1$ for
all $i$, and hence $D$ is invertible.  With this notation, the
\emph{(normalized) graph Laplacian} is given by
\begin{equation*}
\laplacian(\graph) = I - D^{-1/2} A D^{-1/2}.
\end{equation*}
Note that the graph Laplacian $\laplacian = \laplacian(\graph)$ is always
symmetric, positive semidefinite, and satisfies $\laplacian D^{1/2}\onevec =
0$.  Therefore, when the graph is degree-regular ($\mydegree_i = \mydegree$
for all $i \in \vertex$), the standard random walk with self loops on $G$
given by the matrix $\stochmat \defeq I - \frac{\degree}{\degree +
  1}\laplacian$ is doubly stochastic and is valid for our theory.  For
non-degree regular graphs, we need to make a minor modification in order to
obtain a doubly stochastic matrix.

\begin{figure}[t]
  \begin{center}
    \begin{tabular}{cccc}
      \hspace{-.3cm}
      \includegraphics[width=.24\columnwidth]{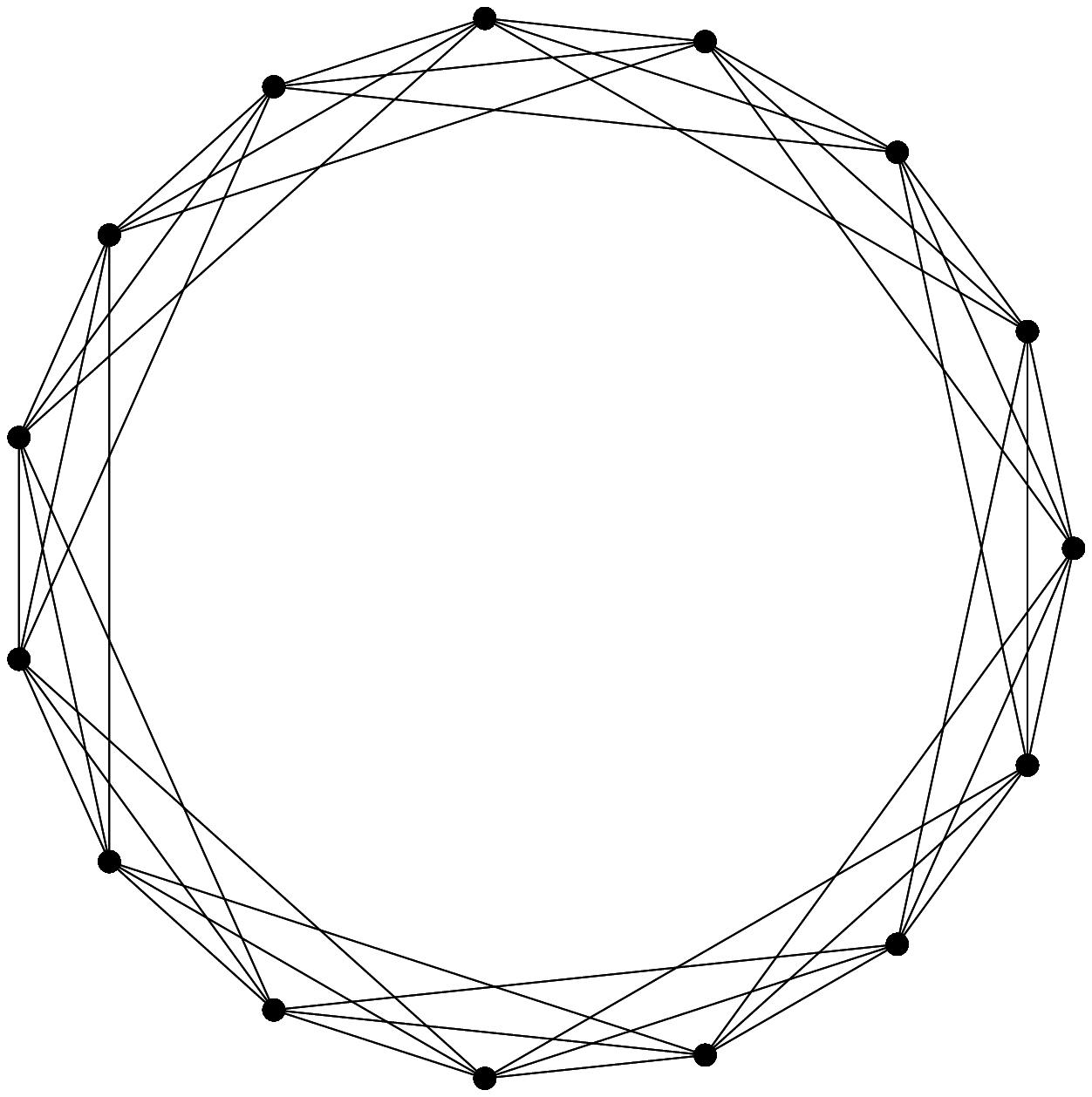} &
      \hspace{-.3cm}
      \includegraphics[width=.24\columnwidth]{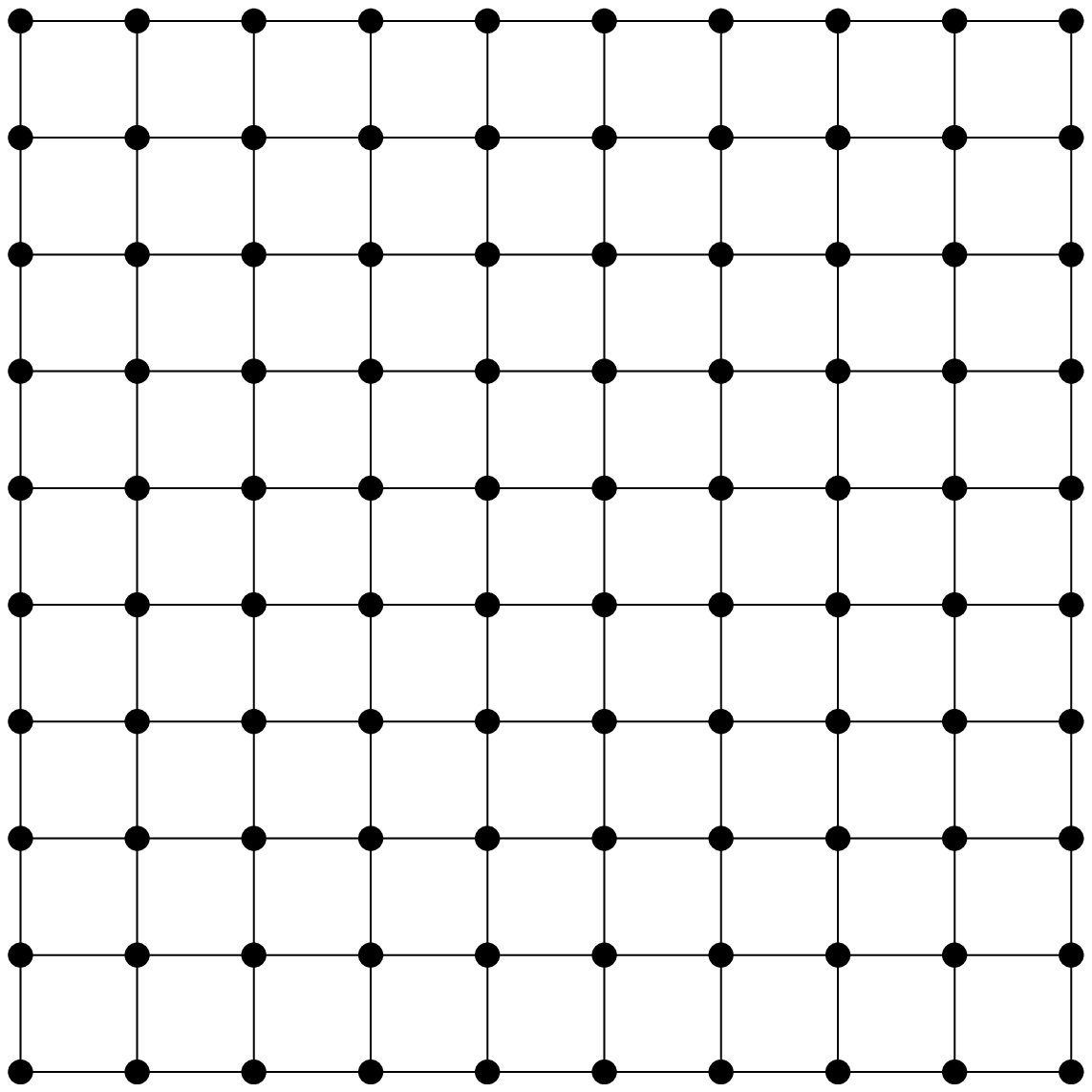} &
      \hspace{-.3cm}
      \includegraphics[width=.24\columnwidth]{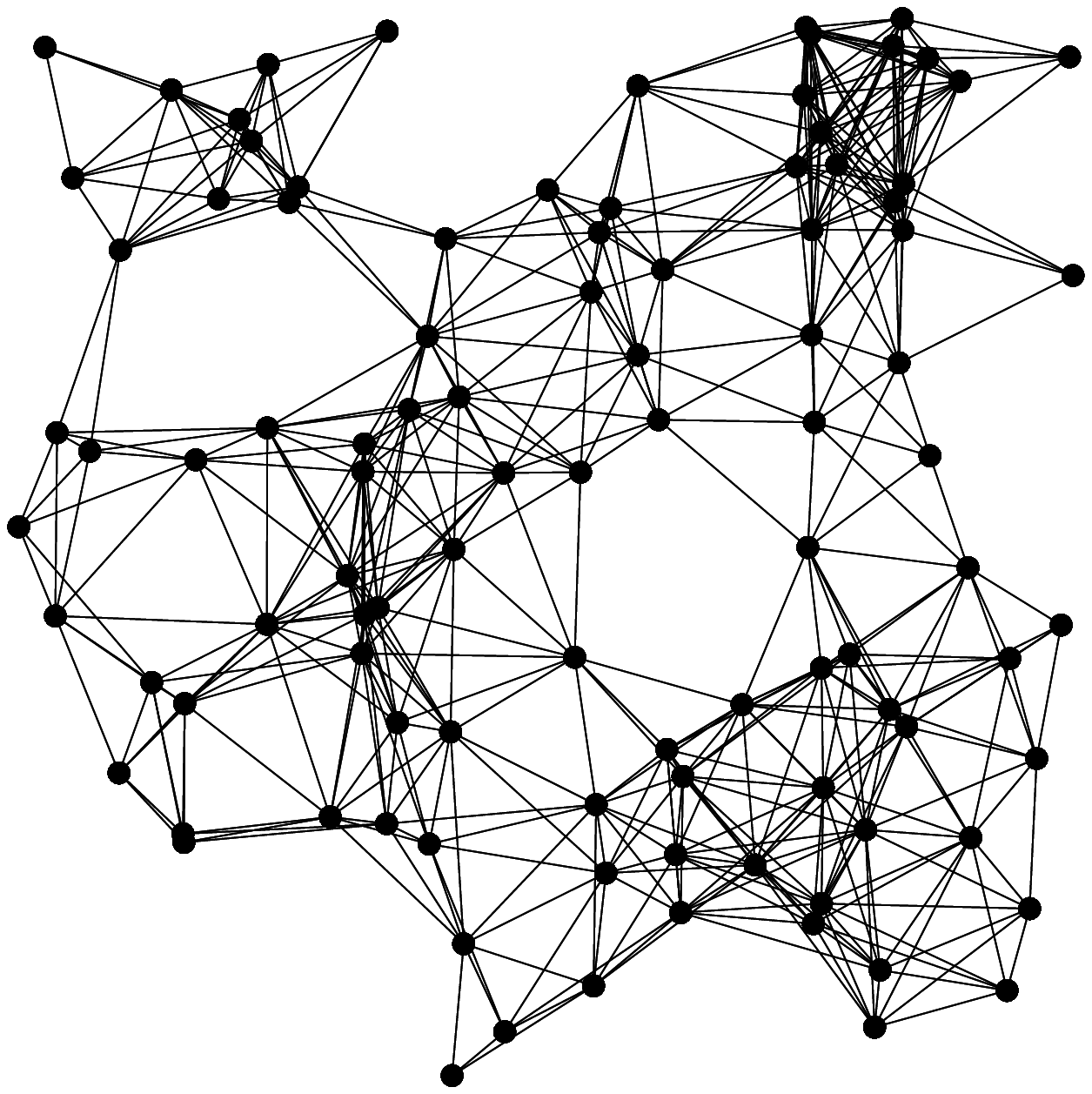} &
      \hspace{-.3cm}
      \includegraphics[width=.24\columnwidth]{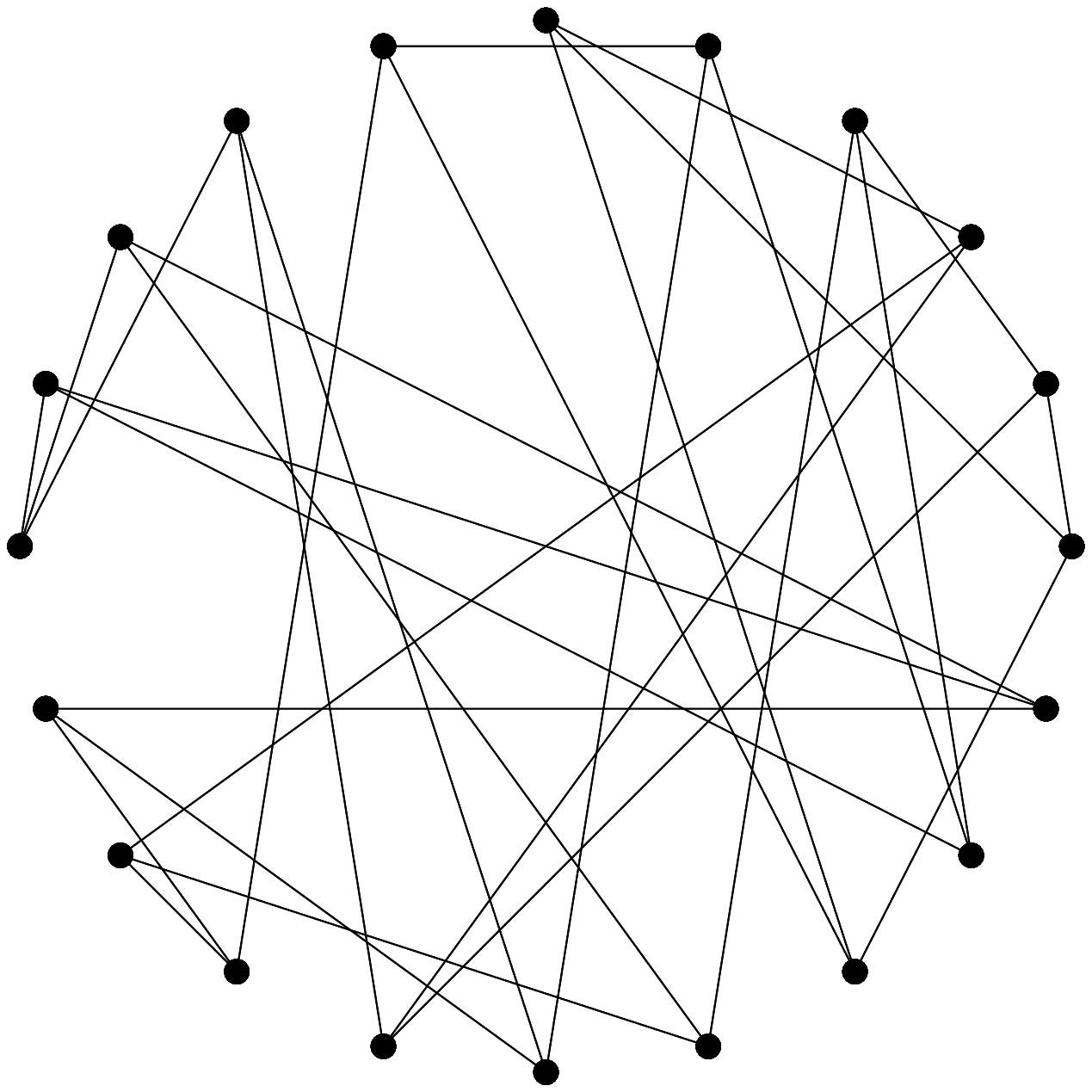} \\
      (a) & (b) & (c) & (d)
    \end{tabular}
    \caption{\label{fig:graph-types} Illustration of some graph
classes of interest in distributed protocols. (a) A $3$-connected
cycle.  (b) Two-dimensional grid with $4$-connectivity, and
non-toroidal boundary conditions. (c) A random geometric graph. (d) A
random 3-regular expander graph.}
  \end{center}
\end{figure}

Letting $\dmax = \max_{i \in \vertex} \mydegree_i$  denote the maximum
degree, we define the modified matrix
\begin{align}
  \label{eqn:mod-stochmat}
  \stochmat_n(\graph)
  & \defeq I - \frac{1}{\dmax+1} \big(D - A \big) \; = \; I -
  \frac{1}{\dmax + 1} D^{1/2} \laplacian D^{1/2}.
\end{align}
This matrix is symmetric by construction, and moreover, $\sum_{j=1}^n
(D_{ij} - A_{ij}) = D_{ii} - \sum_{j=1}^n A_{ij} = 0$ for all $i \in \vertex$,
so it is also doubly stochastic.  Note that if the graph is $\degree$-regular,
then $\stochmat_n(G)$ is the standard choice above.  Modulo a small technical
detail about the ratios of $\dmax$ to $\degree_i$ and the eigenvalue order of
$\stochmat$ (see Sec.~\ref{sec:communication-graph}), plugging
$\stochmat_n(\graph)$ from \eqref{eqn:mod-stochmat} above into
Theorem~\ref{theorem:simple-convergence} immediately relates the convergence
of distributed dual averaging to the spectral properties of the graph
Laplacian, in particular, we have:
\begin{equation}
  \label{eqn:conv-rate-laplacian}
  f(\what{x}_i(T)) - f(x^*) = \order\left(\frac{RL}{\sqrt{T}}
  \frac{\log(Tn)}{\sqrt{\lambda_{n-1}(\laplacian(\graph))}}\right).
\end{equation}
The following result summarizes our conclusions for the choice of stochastic
matrix in \eqref{eqn:mod-stochmat} via \eqref{eqn:conv-rate-laplacian} in
application to different network topologies.

\begin{corollary}
  Under the conditions of Theorem~\ref{theorem:simple-convergence}, we
  have the following convergence rates:
  \begin{enumerate}[(a)]
  \item For $k$-connected paths and cycles,
    \begin{equation*}
    f(\what{x}_i(T)) - f(x^*) = \order\left(\frac{RL}{\sqrt{T}}
    \: \frac{n\log (T n)}{k}\right).
    \end{equation*}
  \item For $k$-connected $\sqrt{n} \times \sqrt{n}$ grids,
    \begin{equation*}
    f(\what{x}_i(T)) - f(x^*) = \order\left(\frac{RL}{\sqrt{T}} \:
    \frac{\sqrt{n} \log(T n)}{k}\right).
    \end{equation*}
  \item For random geometric graphs with connectivity radius $r =
     \Omega(\sqrt{\log^{1+\epsilon} n / n})$ for any $\epsilon > 0$,
    \begin{equation*}
    f(\what{x}_i(T)) - f(x^*) = \order\left(\frac{RL}{\sqrt{T}}
    \: \sqrt{\frac{n}{\log n}}\log (T n)\right)
    \end{equation*}
    with high-probability. 
  \item For expanders with bounded ratio of minimum to maximum node degree,
    \begin{equation*}
    f(\what{x}_i(T)) - f(x^*) = \order\left(\frac{RL}{\sqrt{T}} \: 
    \log(T n)\right).
    \end{equation*}
  \end{enumerate}
  \label{corollary:graph-catalog}
\end{corollary}

Note that up to logarithmic factors, the optimization term in the
convergence rate is always of 
the order $RL/\sqrt{T}$, while the remaining terms vary depending on
the network topology.  Instead of stating convergence rates, in order
to understand scaling issues as a function of network size and
topology, it can be useful to re-state these results in terms of the
number of iterations $\Tspec{\epsilon}{n}$ required to achieve error
$\epsilon$ for network type $\graph$ with $\numnode$ nodes.  As some
special cases, Corollary~\ref{corollary:graph-catalog} implies the
following scalings:
\begin{itemize}
\item for the $1$-connected single cycle graph, we have
  $T_{\operatorname{cycle}}(\epsilon; n) = \order(\numnode^2/\epsilon^2)$.
\item for the two-dimensional grid, we have
  $T_{\operatorname{grid}}(\epsilon; n) =
  \order(\numnode/\epsilon^2)$, and
\item for a bounded degree expander, we have
  $T_{\operatorname{exp}}(\epsilon; \numnode) = \order(1/\epsilon^2)$.
\end{itemize}
\noindent In general, Theorem~\ref{theorem:simple-convergence} implies
that at most
\begin{equation}
  \label{eqn:general-tspec-upper}
  \Tspec{\epsilon}{n} = \order\Big(\frac{1}{\epsilon^2} \cdot
  \frac{1}{1 - \sigma_2(\stochmat_n(\graph))}\Big)
\end{equation}
iterations are required to achieve an $\epsilon$-accurate solution
when using the matrix $\stochmat_n(G)$ previously
defined in~\eqref{eqn:mod-stochmat}.

It is interesting to ask whether the upper bound
\eqref{eqn:general-tspec-upper} from our analysis is actually a sharp result,
meaning that it cannot be improved (up to constant factors).  On one hand, it
is known that (even for centralized optimization algorithms), any subgradient
method requires at least $\Omega\left(\frac{1}{\epsilon^2}\right)$ iterations
to achieve $\epsilon$-accuracy~\cite{NemirovskiYu83}, so that the
$1/\epsilon^2$ term is unavoidable.
The next proposition addresses the complementary issue, namely
whether the inverse spectral gap term is unavoidable for the dual
averaging algorithm.  For the quadratic proximal function $\psi(x) =
\frac{1}{2} \norm{x}_2^2$, the following result establishes a lower
bound on the number of iterations in terms of graph topology and
network structure:

\begin{proposition}
  \label{proposition:tight-spectral-gap}
  Consider the dual averaging algorithm~\eqref{eqn:lazy-unproject}
  and~\eqref{eqn:lazy-project} with quadratic proximal function and
  communication matrix $\stochmat_\numnode(\graph)$. For any graph $G$
  with $\numnode$ nodes, the number of iterations $\Tspec{c}{\numnode}$
  required to achieve a fixed accuracy $c > 0$ is lower bounded as
  \begin{equation*}
  \Tspec{c}{n} = \Omega \biggr(\frac{1}{1 - 
    \sigma_2(\stochmat_n(\graph))} \biggr).
  \end{equation*}
\end{proposition}
\noindent
The proof of this result, given in Section~\ref{sec:lower-bound},
involves constructing a ``hard'' optimization problem and lower
bounding the number of iterations required for our algorithm to solve
it.  In conjunction with Corollary~\ref{corollary:graph-catalog},
Proposition~\ref{proposition:tight-spectral-gap}
implies that our predicted network scaling is sharp.
Indeed, in Section~\ref{sec:simulations}, we show that the theoretical
scalings from Corollary~\ref{corollary:graph-catalog}---namely,
quadratic, linear, and constant in network size $\numnode$---are
well-matched in simulations of our algorithm.

\subsection{Extensions to stochastic communication links}
\label{subsec:stoch-comm}

Our results also extend to the case when the communication matrix
$\stochmat$ is time-varying and random---that is, the matrix
$\stochmat(t)$ is potentially different for each $t$ and randomly
chosen (but it $\stochmat(t)$ still obeys the constraints imposed by $G$).
Such stochastic communication is of interest for a variety of
reasons. If there is an underlying dense network topology, we might
want to avoid communicating along every edge at each round to decrease
communication and network congestion.  For instance, the use of a
gossip protocol~\cite{BoydGhPrSh06}, in which one edge in the network
is randomly chosen to communicate at each iteration, allows for a more
refined trade-off between communication cost and number of iterations.
Communication in real networks also incurs errors due to congestion or
hardware failures, and we can model such errors by a stochastic
process.  

The following theorem provides a convergence result for the case of
time-varying random communication matrices.  In particular, it applies
to sequences $\{x_i(t)\}_{t=0}^\infty$ and $\{z_i(t)\}_{t=0}^\infty$
generated by the dual averaging algorithm with
updates~\eqref{eqn:lazy-unproject} and~\eqref{eqn:lazy-project} with
step size sequence $\{\stepsize(t)\}_{t=0}^\infty$, but in which
$\stochvec_{ij}$ is replaced with $\stochvec_{ij}(t)$.

\begin{theorem}[Stochastic communication]
  \label{theorem:random-P-convergence}
  Let $\{\stochmat(t)\}_{t=0}^\infty$ be an i.i.d.\ sequence of
  doubly stochastic matrices, and define $\lambda_2(\graph)
  \defeq \lambda_2(\E[\stochmat(t)^\top \stochmat(t)])$.
  For any $x^* \in \xdomain$
  and $i \in \vertex$, with probability at least $1 - 1/T$, we have
  \begin{equation*}
    f(\what{x}_i(T)) - f(x^*) \leq \frac{1}{T \stepsize(T)} \prox(x^*) +
    \frac{L^2}{2T} \sum_{t=1}^T \stepsize(t-1) + \frac{3L^2}{T}
    \left(\frac{6 \log(T^2 n)}{1 - \lambda_2(G)} + \frac{1}{T \sqrt{n}}
    + 2 \right) \sum_{t=1}^T \stepsize(t).
  \end{equation*}
\end{theorem}
We provide a proof of the theorem in Section~\ref{sec:random-P-convergence}.
Note that the upper bound from the theorem is valid for any sequence of
non-increasing positive stepsizes $\{\alpha(t)\}_{t=0}^\infty$.  The bound
consists of three terms, with the first growing and the last two shrinking as
the stepsize choice is reduced.  If we assume that $\prox(x^*) \le \radius^2$,
then we can optimize the tradeoff between these competing terms, and we find
that the stepsize sequence $\stepsize(t) \propto
\frac{\radius\sqrt{1 - \lambda_2}}{L \sqrt{t}}$ approximately minimizes the
bound bound in the theorem. This yields the scaling
\begin{equation}
  f(\what{x}_i(T)) - f(x^*) \leq c \, \frac{\radius L}{\sqrt{T}} \cdot
  \frac{\log(Tn)}{\sqrt{1 - \lambda_2(\E[\stochmat(t)^\top\stochmat(t)])}},
  \label{eqn:simplified-stochmat-convergence}
\end{equation}
for a universal constant $c$.  We can also boost the probability with
which this result holds to $1 - 1/T^k$ for any $k>1$---without
modifying the algorithm---at the cost of
incurring a slightly higher constant penalty in the error bound.

The setting of stochastic communication for distributed optimization
was previously considered by Lobel and Ozdaglar~\cite{LobelOz09}.
They established convergence by assuming lower bounds on the entries
of $\stochmat$ whenever two nodes communicated. As a consequence,
their bounds grew exponentially in the number of nodes $n$ in the
network.\footnote{More precisely, inspection of the constant $C$ in
equation~(37) of their paper shows that it is of order $\gamma^{-2(n -
1)}$, where $\gamma$ is the lower bound on non-zero entries of
$\stochmat$, so it is at least $4^{n-1}$.} In contrast, the rates
given here for stochastic communication are directly comparable to the
convergence rates in the previous section for fixed transition
matrices. More specifically, we have inverse dependence on the
spectral gap of the expected network, and consequently polynomial scaling for
any network, as well as faster rates dependent on network structure.

\subsection{Results for stochastic gradient algorithms}

Finally, none of our convergence results rely on the gradients being
correct. Specifically, we can straightforwardly extend our results to the case
of noisy gradients corrupted with zero-mean bounded-variance noise. This
setting is especially relevant in situations such as distributed learning or
wireless sensor networks, when data observed is noisy. Let $\mc{F}_t$ be the
$\sigma$-field containing all information up to time $t$, that is,
$g_i(1),\ldots,g_i(t) \in \mc{F}_t$ and $x_i(1), \ldots, x_i(t+1) \in
\mc{F}_t$ for all $i$. We define a stochastic oracle that provides gradient
estimates satisfying
\begin{equation}
  \E\left[\what{g}_i(t)\mid\mc{F}_{t-1}\right] \in \partial
  f_i(x_i(t))
  ~~~ {\rm and} ~~~
  \E\left[\dnorm{\what{g}_i(t)}^2\mid\mc{F}_{t-1}\right]
  \leq L^2.
  \label{eqn:stochoracle}
\end{equation}
As a special case, this model includes an additive noise oracle that
takes an element of the subgradient $\partial f_i(x_i(t))$ and adds to
it bounded variance zero-mean noise. Theorem~\ref{theorem:stochastic-gradient}
gives our result in the case of stochastic gradients. We give the
proof and further discussion in Section~\ref{sec:stoch-gradient},
noting that because we adapt the dual averaging algorithm, the
analysis follows quite cleanly from the earlier analysis for the
previous three theorems.
\begin{theorem}[Stochastic gradient updates]
  \label{theorem:stochastic-gradient}
  Let the sequence $\{x_i(t)\}_{t=1}^\infty$ be as in
  Theorem~\ref{theorem:master-convergence}, except that at each round
  of the algorithm agent $i$ receives a vector $\what{g}_i(t)$ from an
  oracle satisfying condition~\eqref{eqn:stochoracle}. For each $i \in
  V$, we have
  \begin{equation*}
    \E \Big[f(\what{x_i}(T))\Big] - f(x^*) \le \frac{1}{T\stepsize(T)}
    \prox(x^*) + \frac{8 L^2}{T} \sum_{t=1}^T
    \stepsize(t-1) + \frac{3 L^2}{T} \frac{\log (T
      \sqrt{n})}{ 1 - \sigma_2(\stochmat)}
    \sum_{t=1}^T \stepsize(t).
  \end{equation*}
  If we assume in addition that $\xdomain$ has finite radius $\radius \defeq
  \sup_{x \in \xdomain} \norm{x - x^*}$ and that $\dnorm{\what{g}_i(t)}
  \le L$, then with probability at least $1 - \delta$,
  \begin{equation*}
    f(\what{x}_i(T)) - f(x^*) \leq \frac{1}{T\stepsize(T)} \prox(x^*) +
    \frac{8 L^2}{T} \sum_{t=1}^T \stepsize(t-1)
    + \frac{3L^2}{T} \frac{\log (T \sqrt{n})}{1 -
      \sigma_2(\stochmat)} \sum_{t=1}^T\stepsize(t)
    + 8 L \radius \sqrt{\frac{\log \frac{1}{\delta}}{T}}.
  \end{equation*}
  If we further assume that the gradient estimates $\what{g}_i(t)$ are
  uncorrelated given $\mc{F}_{t-1}$, then with probability at least $1 -
  \delta$,
  \begin{equation*}
    f(\what{x}_i(T)) - f(x^*) \le \frac{1}{T \stepsize(T)} \prox(x^*)
    + \frac{8L^2}{T} \sum_{t=1}^T \stepsize(t - 1)
    + \frac{3 L ^2}{T} \frac{\log (T \sqrt{n})}{1 - \sigma_2(\stochmat)}
    \sum_{t=1}^T \stepsize(t)
    + \frac{3 L\radius \log \frac{1}{\delta}}{T}
    + 4 L\radius  \sqrt{\frac{\log \frac{1}{\delta}}{nT}}.
  \end{equation*}
\end{theorem}
\noindent
As with the case of stochastic communication covered by
Theorem~\ref{theorem:random-P-convergence}, it should be clear that by
choosing the stepsize $\stepsize(t) \propto \frac{\radius \sqrt{1 -
    \sigma_2(\stochmat)}}{L \sqrt{t}}$, we have essentially the same
optimization error guarantee as the bound
\eqref{eqn:simplified-stochmat-convergence}, but with
$\lambda_2(\E[\stochmat(t)^\top\stochmat(t)])$ replaced by
$\sigma_2(\stochmat)$.

\section{Related Work}
\label{sec:related-work}

Having stated and discussed our main results in the previous section, we can
now more explicitly compare the results in this paper to those in previous
work. Our aim here is to give a clear understanding of how our algorithm and
results relate to and, in many cases, improve upon prior results.
Specifically, with the results of Theorem~\ref{theorem:simple-convergence} and
Corollary~\ref{corollary:graph-catalog} in hand, we can more directly compare
our results to other work.

As discussed in the introduction, other researchers have designed algorithms
for solving the problem~\eqref{eqn:objective}. Most previous
work~\cite{LobelOz09,NedicOz09,NedicOlOzTs09,RamNeVe10} studies convergence of
a (projected) gradient method in which each node $i$ in the network maintains
$x_i(t) \in \xdomain$, and at time $t$ performs the update
\begin{equation}
  x_i(t + 1) = \argmin_{x \in \xdomain} \Big\{\half \ltwo{ \sum_{j \in
      \neighbor(i)} \stochmat_{ji} x_j(t) - \stepsize g_i(t)}^2 \Big\} 
  \label{eqn:nedic-alg}
\end{equation}
for $g_i(t) \in \partial f_i(x_i(t))$. With the update~\eqref{eqn:nedic-alg},
Corollary~5.5 in the paper~\cite{RamNeVe10} shows that
\begin{equation*}
  f(\what{x}_i(T)) - f(x^*) =
  \order\left(\frac{\stepsize n^3 \radius^2}{T} + \stepsize L^2 \right)
\end{equation*}
(we use our notation and assumptions from
Theorem~\ref{theorem:simple-convergence}).  The above bound is minimized by
setting the stepsize $\stepsize \propto \frac{L}{n^{3/2} \radius \sqrt{T}}$,
giving convergence rate $\order(n^{3/2} \radius L / \sqrt{T})$.  It is clear
that this convergence rate is substantially slower than all the rates in
Corollary~\ref{corollary:graph-catalog}.

The distributed dual averaging
algorithm~\eqref{eqn:lazy-unproject}--\eqref{eqn:lazy-project} is quite
different from the update~\eqref{eqn:nedic-alg}. The use of the proximal
function $\prox$ allows us to address problems with non-Euclidean geometry,
which is useful, for example, for very high-dimensional problems or where the
domain $\xdomain$ is the simplex~(e.g.~\cite[Chapter 3]{NemirovskiYu83}). The
differences between the algorithms become more pronounced in the
analysis. Since we use dual averaging, we can avoid some technical
difficulties introduced by the projection step in the
update~\eqref{eqn:nedic-alg}. Precisely because of this technical issue,
earlier works~\cite{NedicOz09,LobelOz09} studied unconstrained optimization,
and the averaging in $z_i(t)$ seems essential to the faster rates our approach
achieves as well as the ease with which we can extend our results to
stochastic settings.

In other related work, Johansson et al.~\cite{JohanssonRaJo09}
establish network-dependent rates for Markov incremental gradient
descent (MIGD), which maintains a single vector $x(t)$ at all times. A
token $i(t)$ determines an active node at time $t$, and at time step
$t + 1$ the token moves to one of its neighbors $j \in
\neighbor(i(t))$, each with probability $\stochmat_{j i(t)}$. Letting
$g_{i(t)}(t) \in \partial f_{i(t)}(x(t))$, the update is
\begin{equation}
  \label{eqn:migd}
  x(t + 1) = \argmin_{x \in \xdomain} \Big\{\half \ltwo{x(t) -
    \stepsize g_{i(t)}(t)}^2 \Big\}.
\end{equation}
Johansson et al.\ show that with optimal setting of $\stepsize$ and symmetric
transition matrix $\stochmat$, MIGD has convergence rate $\order(\radius L
\max_i \sqrt{\frac{n \Gamma_{ii}}{T}})$, where $\Gamma$ is the return time
matrix \mbox{$\Gamma = (I - \stochmat + \onevec \onevec^\top / n)^{-1}$}. In
this case, let $\lambda_i(\stochmat) \in [-1, 1]$ denote the $i$th eigenvalue
of $\stochmat$. The eigenvalues of $\Gamma$ are thus $1$ and $1 / (1 -
\lambda_i(\stochmat))$ for $i > 1$, and so we have
\begin{align*}
  n \max_{i=1, \ldots, n} \Gamma_{ii}
  \ge \tr(\Gamma) = 1 +
  \sum_{i=2}^n\frac{1}{1 - \lambda_i(\stochmat)}
  > \max\left\{\frac{1}{1 - \lambda_2(\stochmat)}, \frac{1}{1 -
    \lambda_n(\stochmat)}\right\} = \frac{1}{1 - \sigma_2(\stochmat)}.
\end{align*}
Consequently, the bound in Theorem~\ref{theorem:simple-convergence} is never
weaker, and for certain graphs, our results are substantially tighter, as
shown in Corollary~\ref{corollary:graph-catalog}.  For $d$-dimensional grids
(where $d \ge 2$) we have \mbox{$T(\epsilon; \numnode) = \order(n^{2/d} /
  \epsilon^2)$,} whereas MIGD scales as $T(\epsilon; \numnode) =
\order(n/\epsilon^2)$. For well-connected graphs, such as expanders and the
complete graph, the MIGD algorithm scales as $T(\epsilon; \numnode) = \order(n
/ \epsilon^2)$, essentially a factor of $n$ worse than our results.

\section{Basic convergence analysis for distributed dual averaging}
\label{sec:master-convergence}

In this section, we prove convergence of the distributed algorithm
based on the updates~\eqref{eqn:lazy-unproject}
and~\eqref{eqn:lazy-project}.  We begin in Section~\ref{SecSetupFour}
by defining some auxiliary quantities and establishing lemmas useful
in the proof, and we prove Theorem~\ref{theorem:master-convergence}
in Section~\ref{SecProofMaster}.

\subsection{Setting up the analysis}
\label{SecSetupFour}

Using techniques related to those used in past work~\cite{NedicOz09},
we establish convergence via two auxiliary sequences, given by
\begin{equation}
  \bar{z}(t) \defeq \frac{1}{n} \sum_{i=1}^n z_i(t) ~~~ {\rm and} ~~~
  y(t) \defeq \Pi_\xdomain^\prox(-\bar{z}(t),\stepsize).
  \label{eqn:ydefn}
\end{equation}
We begin by showing that the average sum of gradients $\bar{z}(t)$
evolves in a very simple way.  In particular, we have
\begin{align*}
  \bar{z}(t + 1) & = \frac{1}{n} \sum_{i=1}^n \sum_{j=1}^n
  \big(\stochvec_{ij} z_j(t) + g_i(t) \big)
\end{align*}
Consider the right-hand side above, let $Z(t) = [z_1(t) ~
\cdots ~ z_n(t)]$ be the matrix of vectors $z_i$, and denote
$\stochmat = [\stochvec_1 ~ \cdots ~ \stochvec_n]$. Since
the matrix $\stochmat$ is doubly stochastic, we have
\begin{equation*}
\frac{1}{n} \sum_{i=1}^n \sum_{j=1}^n \stochvec_{ij} z_j(t) \; = \;
\frac{1}{n} Z(t) \stochmat \onevec \; = \; \frac{1}{n} Z(t) \onevec
\;= \; \bar{z}(t),
\end{equation*}
which yields the evolution
\begin{equation}
  \label{eqn:mean-z-update}
  \bar{z}(t + 1) = \bar{z}(t) + \ninv \sum_{j=1}^n g_j(t).
\end{equation}
Consequently, the (negative of the) averaged dual sequence
$\{\zb(t)\}_{t=0}^\infty$ evolves almost like standard subgradient descent on
the function $f(x) = \sum_{i=1}^n f_i(x)/n$, the only difference being
$g_i(t)$ is a subgradient at $x_i(t)$ (which need not be the same as the
subgradient $g_j(t)$ at $x_j(t)$).  The simple
evolution~\eqref{eqn:mean-z-update} of the averaged dual sequence allows us to
avoid difficulties with the non-linearity of projection that have been
challenging in earlier work.

Before proceeding with the proof of
Theorem~\ref{theorem:master-convergence}, we state two useful results
regarding the convergence of the standard dual averaging algorithm,
though we defer their proofs to Appendix~\ref{app:lazy}.  We begin by
giving a convergence guarantee for the single-objective form of the
dual averaging algorithm. Let $\{g(t)\}_{t=1}^\infty \subset \R^d$ be
an arbitrary sequence of vectors, and consider the sequence
$\{x(t)\}_{t=1}^\infty$ defined by
\begin{equation}
  \label{eqn:ftrl-linear}
  x(t + 1) \; \defeq \; \argmin_{x \in \xdomain} \bigg\{ \sum_{s=1}^t
  \<g(s), x\> + \frac{1}{\stepsize(t)} \prox(x) \bigg\} =
  \Pi_\xdomain^\prox \bigg(\sum_{s=1}^t g(s), \stepsize(t) \bigg).
\end{equation}
\begin{lemma}
  \label{lemma:ftrl-linear}
  For any non-increasing sequence $\{\stepsize(t)\}_{t=0}^\infty$ of
  positive stepsizes, and for any $x^* \in \xdomain$, we have
  \begin{equation*}
  \sum_{t=1}^T \<g(t), x(t) - x^*\> \leq
  \half \sum_{t=1}^T \stepsize(t - 1) \dnorm{g(t)}^2
  + \frac{1}{\stepsize(T)} \prox(x^*).
  \end{equation*}
\end{lemma}
\noindent

\noindent Next we state a lemma that allows us to restrict our analysis to the
easier to analyze centralized sequence $\{y(t)\}_{t=0}^\infty$ from
\eqref{eqn:ydefn}:
\begin{lemma}
  \label{lemma:transfer-y-to-x}
  Consider the sequences $\{x_i(t)\}_{t=1}^\infty$,
  $\{z_i(t)\}_{t=0}^\infty$, and $\{y(t)\}_{t=0}^\infty$ defined
  according to equations~\eqref{eqn:lazy-unproject},
  ~\eqref{eqn:lazy-project}, and~\eqref{eqn:ydefn}. Recall
  that each $f_i$ is $L$-Lipschitz. For each $i \in
  \vertex$, we have
  \begin{align*}
    \sum_{t=1}^Tf(x_i(t)) - f(x^*) \leq \sum_{t=1}^T f(y(t)) - f(x^*) +
    L\sum_{t=1}^T \stepsize(t)\dnorm{\bar{z}(t) - z_i(t)}.
  \end{align*}
  Similarly, with the definitions $\what{y}(T) \defeq \frac{1}{T}\sum_{t=1}^T
  y(t)$ and $\what{x}_i(T) \defeq \frac{1}{T} \sum_{t=1}^T x_i(t)$,
  we have
  \begin{align*}
    f(\what{x}_i(T)) - f(x^*) & \leq f(\what{y}(T)) - f(x^*) +
    \frac{L}{T}\sum_{t=1}^T \stepsize(t)\dnorm{\bar{z}(t) - z_i(t)}. 
  \end{align*}
\end{lemma}
\noindent Equipped with these tools, we now turn the proof of
Theorem~\ref{theorem:master-convergence}.

\subsection{Proof of Theorem~\ref{theorem:master-convergence}}
\label{SecProofMaster}
Our proof is based on analyzing the sequence $\{y(t)\}_{t=0}^\infty$.
Given an arbitrary $x^* \in \xdomain$, we have
\begin{align}
\sum_{t=1}^T f(y(t)) - f(x^*) & = \sum_{t=1}^T \ninv\sum_{i=1}^n
 f_i(x_i(t)) - f(x^*) + \sum_{t=1}^T \ninv\sum_{i=1}^n \left[f_i(y(t))
 - f_i(x_i(t))\right] \nonumber \\ & \le \ninv\sum_{t=1}^T
 \sum_{i=1}^n f_i(x_i(t)) - f(x^*) + \sum_{t=1}^T \sum_{i=1}^n
 \frac{L}{n} \norm{y(t) - x_i(t)},
 \label{eqn:dist-lazy-to-bound}
\end{align}
where the inequality follows by the $L$-Lipschitz condition on $f_i$.

Let $g_i(t) \in \partial f_i(x_i(t))$ be a subgradient of $f_i$ at
$x_i(t)$.  Using convexity, we have the bound
\begin{equation}
  \ninv\sum_{t=1}^T \sum_{i=1}^n f_i(x_i(t)) - f_i(x^*) \le
  \ninv\sum_{t=1}^T \sum_{i=1}^n \<g_i(t), x_i(t) - x^*\>.
  \label{eqn:first-order-dist-bound}
\end{equation}
Breaking up the right hand side of~\eqref{eqn:first-order-dist-bound}
into two pieces, we obtain
\begin{align}
  \sum_{i=1}^n \<g_i(t), x_i(t) - x^*\> & = \sum_{i=1}^n \<g_i(t), y(t)
  - x^*\> + \sum_{i=1}^n \<g_i(t), x_i(t) - y(t)\>.
  \label{eqn:first-order-y-replacement}
\end{align}
By definition of the updates for $\bar{z}(t)$ and $y(t)$, we have
\begin{equation*}
  y(t) = \argmin_{x \in \xdomain}
  \bigg\{\ninv\sum_{s=1}^{t-1}\sum_{i=1}^n\<g_i(s), x\> +
  \frac{1}{\stepsize(t)}\prox(x)\bigg\}.
\end{equation*}
Thus, we see that the first term in the
decomposition~\eqref{eqn:first-order-y-replacement} can be written in
the same way as the bound in Lemma~\ref{lemma:ftrl-linear}, and as a
consequence, we have the bound
\begin{align}
  \ninv \sum_{t=1}^T \<\sum_{i=1}^n g_i(t), y(t) - x^*\>
  & \le \half \sum_{t=1}^T \stepsize(t-1)
  \dnormb{\ninv \sum_{i=1}^n g_i(t)}^2
  + \frac{1}{\stepsize(T)} \prox(x^*) \nonumber \\
  & \le \frac{L^2}{2}
  \sum_{t=1}^T \stepsize(t-1) + \frac{1}{\stepsize(T)} \prox(x^*).
  \label{eqn:dist-lazy-barz-bound}
\end{align}
It remains to control the final two terms in the
bounds~\eqref{eqn:dist-lazy-to-bound} and
\eqref{eqn:first-order-y-replacement}.  Since $\dnorm{g_i(t)} \le L$
by assumption, we have
\begin{align*}
  \lefteqn{\sum_{t=1}^T \sum_{i=1}^n \frac{L}{n} \norm{y(t) - x_i(t)} +
    \ninv\sum_{t=1}^T \sum_{i=1}^n \<g_i(t), x_i(t) - y(t)\>} \\
  & \le \frac{2L}{n} \sum_{t=1}^T \sum_{i=1}^n \norm{y(t) - x_i(t)} \\
  & = \frac{2 L}{n} \sum_{t=1}^T \sum_{i=1}^n
  \norm{\Pi^\prox_\xdomain(-\bar{z}(t), \stepsize(t)) -
  \Pi^\prox_\xdomain(-z_i(t), \stepsize(t))}.
\end{align*}
By the $\stepsize$-Lipschitz continuity of the projection operator
$\Pi_\xdomain^\prox(\cdot, \stepsize)$ (see
Appendix~\ref{Applemma:projection}), we have
\begin{align*}
  \frac{2 L}{n} \sum_{t=1}^T \sum_{i=1}^n
  \norm{\Pi^\prox_\xdomain(\bar{z}(t), \stepsize(t)) -
    \Pi^\prox_\xdomain(z_i(t), \alpha)}
  & \le \frac{2 L}{n} \sum_{t=1}^T
  \sum_{i=1}^n \stepsize(t)\dnorm{\bar{z}(t) - z_i(t)}.
\end{align*}
Combining this bound with~\eqref{eqn:dist-lazy-to-bound}
and~\eqref{eqn:dist-lazy-barz-bound} yields the running sum bound
\begin{equation}
  \label{eqn:running-sum-bound}
  \sum_{t=1}^T \big[f(y(t)) - f(x^*)\big]
  \le
  \frac{1}{\stepsize(T)} \prox(x^*) + \frac{L^2}{2} \sum_{t=1}^T \stepsize(t-1)
  + \frac{2L}{n} \sum_{t=1}^T \sum_{j=1}^n \stepsize(t) \dnorm{\bar{z}(t)
    - z_j(t)}.
\end{equation}
Applying Lemma~\ref{lemma:transfer-y-to-x} to \eqref{eqn:running-sum-bound}
gives that $\sum_{t=1}^T [f(x_i(t)) - f(x^*)]$ is upper bounded by
\begin{equation*}
  \frac{1}{\stepsize(T)} \prox(x^*) + \frac{L^2}{2}
  \sum_{t=1}^T\stepsize(t-1) + \frac{2 L}{n} \sum_{t=1}^T
  \sum_{j=1}^n \stepsize(t)\dnorm{\bar{z}(t) - z_j(t)} +
  L\sum_{t=1}^T \stepsize(t)\dnorm{\bar{z}(t) - z_i(t)}.
\end{equation*}
Dividing both sides by $T$ and using convexity of $f$ yields the
bound~\eqref{eqn:master-convergence}.

\section{Convergence rates, spectral gap, and network topology}
\label{sec:fixed-p-convergence}

In this section, we will give concrete convergence rates for the
distributed dual averaging algorithm based on the mixing time of a
random walk according to the doubly stochastic matrix $\stochmat$.
The understanding of the dependence of our convergence rates in terms
of the underlying network topology is crucial, because it can provide
important cues to the system administrator in a clustered computing
environment or for the locations and connectivities of sensors in a
sensor network.  We begin in Section~\ref{SecProofSimpleConv} with
the proof of Theorem~\ref{theorem:simple-convergence}.  In
Section~\ref{sec:communication-graph}, we prove the graph-specific
convergence rates stated in Corollary~\ref{corollary:graph-catalog},
whereas Section~\ref{sec:lower-bound} contains a proof of the lower
bound stated in Proposition~\ref{proposition:tight-spectral-gap}.

Throughout this section, we adopt the following notational
conventions. For an $n\times n$ matrix $B$, we call its singular
values $\sigma_1(B) \geq \sigma_2(B) \geq \dots \geq \sigma_n(B) \geq
0$. For a real symmetric $B$, we use $\lambda_1(B) \ge \lambda_2(B)
\ge \ldots \ge \lambda_n(B)$ to denote the $n$ real eigenvalues of
$B$. We let $\Delta_n = \{x \in \R^n \, \mid \, x \succeq 0,
\sum_{i=1}^n x_i = 1 \}$ denote the $n$-dimensional probability
simplex.  We make frequent use of the following standard inequality:
for any positive integer $t = 1, 2, \ldots$ and any $x \in \Delta_n$,
\begin{equation}
  \label{eqn:tv-bound}
  \tvnorm{\stochmat^tx - \onevec/n}
  = \half \lone{\stochmat^tx - \onevec/n}
  \le \half \sqrt{n}\ltwo{\stochmat^tx - \onevec/n}
  \le \half \sigma_2(\stochmat)^t \sqrt{n}.
\end{equation}
For a brief review of the relevant standard Perron-Frobenius and
matrix theory, we refer the reader to Appendix~\ref{app:markovchain}.

\subsection{Proof of Theorem~\ref{theorem:simple-convergence}}
\label{SecProofSimpleConv}

We focus on controlling the network error term in the
bound~\eqref{eqn:master-convergence}, namely the quantity
\begin{align*}
  \frac{L}{n}\sum_{t=1}^T \sum_{i=1}^n \stepsize(t)
  \dnorm{\bar{z}(t) - z_i(t)}.
\end{align*}
Define the matrix $\Phi(t, s) = \stochmat^{t - s + 1}$ (in the sequel
we allow the stochastic matrix $\stochmat$ to change as a function
of time). Let $[\Phi(t, s)]_{ji}$ be the $j$th entry of the $i$th
column of $\Phi(t, s)$. Then via a bit of algebra, we can write
\begin{equation}
  \label{eqn:evolve-z-phi}
  z_i(t + 1) = \sum_{j=1}^n [\Phi(t, s)]_{ji} z_j(s)
  + \sum_{r = s + 1}^t \bigg(\sum_{j=1}^n [\Phi(t, r)]_{ji}
  g_j(r - 1) \bigg) + g_i(t).
\end{equation}
Clearly the above reduces to the standard update
\eqref{eqn:lazy-unproject} when $s = t$.
Since $\bar{z}(t)$ evolves simply as in \eqref{eqn:mean-z-update}, we assume 
that $z_i(0) = \bar{z}(0)$ to avoid notational clutter---we can simply start
with $z_i(0) = 0$---and use \eqref{eqn:evolve-z-phi} to see
\begin{equation}
  \label{eqn:bar-z-zi-diff}
  \bar{z}(t) - z_i(t) = \sum_{s=1}^{t-1} \sum_{j=1}^n
  (1/n - [\Phi(t - 1, s)]_{ji}) g_j(s - 1) +
  \bigg(\frac{1}{n} \sum_{j=1}^n (g_j(t-1) - g_i(t - 1))\bigg).
\end{equation}
We use the fact that $\dnorm{g_i(t)} \le L$ for all $i$ and $t$
and \eqref{eqn:bar-z-zi-diff} to see that
\begin{align}
  \dnorm{\bar{z}(t) - z_i(t)}
  & = \dnormb{\sum_{s=1}^{t-1} \sum_{j=1}^n
    (1/n - [\Phi(t - 1, s)]_{ji}) g_j(s - 1)
    + \bigg(\frac{1}{n} \sum_{j=1}^n g_j(t - 1) - g_i(t - 1)
    \bigg)} \nonumber \\
  & \le \sum_{s=1}^{t-1} \sum_{j=1}^n \dnorm{g_j(s  -1)}
  |(1 / n) - [\Phi(t - 1, s)]_{ji}| + \ninv
  \sum_{i=1}^n \dnorm{g_j(t-  1) - g_i(t - 1)} \nonumber \\
  & \le \sum_{s=1}^{t-1} L \lone{[\Phi(t - 1, s)]_i - \onevec/n}
  + 2L. \label{eqn:intermediate-bar-z-zi-bound}
\end{align}

Now we break the sum in \eqref{eqn:intermediate-bar-z-zi-bound} into two terms
separated by a cutoff point $\what{t}$. The first term consists of
``throwaway'' terms, that is, timesteps $s$ for which the Markov chain with
transition matrix $\stochmat$ has not mixed, while the second consists of
steps $s$ for which $\lone{[\Phi(t - 1, s)]_i - \onevec/n}$ is small. Note
that the indexing on $\Phi(t - 1, s) = \stochmat^{t - s + 1}$ implies that for
small $s$, $\Phi(t - 1, s)$ is close to uniform.  From \eqref{eqn:tv-bound},
$\lone{[\Phi(t, s)]_j - \onevec/n} \le \sqrt{n} \sigma_2(\stochmat)^{t - s +
  1}$. Hence, if
\begin{equation*}
  t - s \ge \frac{\log \epsilon^{-1}}{\log
    \sigma_2(\stochmat)^{-1}} - 1
  \quad {\rm we~immediately~have} \quad
  \lone{[\Phi(t, s)]_j -
    \onevec/n} \le \sqrt{n} \epsilon
\end{equation*}
Thus, by setting $\epsilon^{-1} = T \sqrt{n}$, for $t - s + 1 \ge \frac{\log(T
  \sqrt{n})}{\log \sigma_2(\stochmat)^{-1}}$, we have
\begin{equation}
  \label{eqn:fixed-P-phi-converge}
  \lone{[\Phi(t, s)]_j - \onevec/n} \le \frac{1}{T}.
\end{equation}
For larger $s$, we simply have $\lone{[\Phi(t, s)]_j - \onevec/n} \le 2$.  The
above suggests that we split the sum at $\what{t} = \frac{\log T\sqrt{n}}{\log
  \sigma_2(\stochmat)^{-1}}$. We break apart the sum in
\eqref{eqn:intermediate-bar-z-zi-bound} and use
\eqref{eqn:fixed-P-phi-converge} to see that since $t - 1 - (t - \what{t}) =
\what{t}$ and there are at most $T$ steps in the summation,
\begin{align}
  \dnorm{\bar{z}(t) - z_i(t)}
  & \le L\sum_{s = t - \what{t}}^{t-1}
  \lone{\Phi(t - 1, s) e_i - \onevec/n}
  + L \sum_{s = 1}^{t - 1 - \what{t}} \lone{\Phi(t - 1, s)
    e_i - \onevec/n} + 2  L \nonumber \\
  & \le 2 L \frac{\log(T \sqrt{n})}{\log \sigma_2(\stochmat)^{-1}}
  + 3L
  \le 2 L \frac{\log (T \sqrt{n})}{1 - \sigma_2(\stochmat)}
  + 3L.
  \label{eqn:fixed-p-convergence}
\end{align}
The last inequality follows from the concavity of $\log(\cdot)$, since $\log
\sigma_2(\stochmat)^{-1} \ge 1 - \sigma_2(P)$.

Combining \eqref{eqn:fixed-p-convergence} with the running sum bound in
\eqref{eqn:running-sum-bound} of the proof of the basic theorem,
Theorem~\ref{theorem:master-convergence}, we immediately see that for $x^* \in
\xdomain$,
\begin{equation}
  \sum_{t=1}^T f(y(t)) - f(x^*) \le \frac{1}{\stepsize(T)} \prox(x^*)
  + \frac{L^2}{2} \sum_{t=1}^T \stepsize(t-1)
  + 6 L^2 \sum_{t=1}^T \stepsize(t)
  + 4 L^2 \frac{\log(T \sqrt{n})}{
    1 - \sigma_2(\stochmat)}\sum_{t=1}^T \stepsize(t).
  \label{eqn:fixed-p-to-cleanup}
\end{equation}
Appealing to Lemma~\ref{lemma:transfer-y-to-x} allows us to obtain the same
result on the sequence $x_i(t)$ with slightly worse constants.  Note that
$\sum_{t=1}^T t^{-1/2} \le 2\sqrt{T} - 1$. Thus, using the assumption that
$\prox(x^*) \le \radius^2$, using convexity to bound $f(\what{y}(T)) \le
\frac{1}{T} \sum_{t=1}^T f(y(t))$ (and similarly for $\what{x}_i(T)$), and
setting $\stepsize(t)$ as in the statement of the theorem completes the proof.

\subsection{Proof of Corollary~\ref{corollary:graph-catalog}}
\label{sec:communication-graph}

The corollary is based on bounding the spectral gap of the matrix
$\stochmat_n(\graph)$ from equation~\eqref{eqn:mod-stochmat}.
\begin{lemma}
  \label{lemma:stochmat-laplacian}
  The matrix $\stochmat_n(\graph)$ satisfies the bound
  \begin{equation*}
    \sigma_2(\stochmat_n(\graph)) \leq \max \Big\{ 1 -
    \frac{\min_i\degree_i}{\dmax + 1}\lambda_{n-1}(\laplacian),
    \frac{\dmax}{\dmax+1}\lambda_1(\laplacian)-1 \Big\}.
  \end{equation*}
\end{lemma}
\begin{proof}
By a theorem of Ostrowski on congruent matrices (cf. Theorem
4.5.9,~\cite{HornJo85}), we have
  \begin{equation}
    \lambda_k(D^{1/2} \laplacian D^{1/2})
    \in \left[\min_i \degree_i \lambda_k(\laplacian),
      \max_i \degree_i \lambda_k(\laplacian) \right].
    \label{eqn:ostrowski}
  \end{equation}
  Since $\laplacian D^{1/2}\onevec = 0$, we have $\lambda_n(\laplacian) = 0$,
  and so it suffices to focus on $\lambda_1(D^{1/2} \laplacian D^{1/2})$
  and $\lambda_{n-1}(D^{1/2} \laplacian D^{1/2})$. From the
  definition~\eqref{eqn:mod-stochmat}, the eigenvalues of $\stochmat$ are of
  the form $1 - (\dmax + 1)^{-1} \lambda_k (D^{1/2} \laplacian
  D^{1/2})$. The bound~\eqref{eqn:ostrowski} coupled with the
  fact that all the eigenvalues of $\laplacian$ are non-negative implies
  that $\sigma_2(\stochmat) = \max_{k < n} \big\{ \big|1 - (\dmax +
  1)^{-1} \lambda_k(D^{1/2} \laplacian D^{1/2}) \big| \big\}$ is upper
  bounded by the larger of
  \begin{equation*}
    1 - \frac{\dmin}{\dmax + 1} \lambda_{n-1}(\laplacian) ~~~ {\rm
      and} ~~~ \frac{\dmax}{\dmax + 1} \lambda_1(\laplacian) - 1.
    % \label{eqn:bound-lambda-laplacian}
    \qedhere
  \end{equation*}
\end{proof}
Much of spectral graph theory is devoted to bounding
$\lambda_{n-1}(\laplacian)$ sufficiently far away from zero, and
Lemma~\ref{lemma:stochmat-laplacian} allows us to conveniently leverage
such results for bounding the convergence rate of our algorithm.

Note that computing the upper bound in Lemma~\ref{lemma:stochmat-laplacian}
requires controlling both $\lambda_{n-1}(\laplacian)$ and
$\lambda_1(\laplacian)$.  In order to circumvent this complication, we
use the well-known idea of a ``lazy'' random
walk~\cite{Chung98,LevinPeWi08}, in which we replace $\stochmat$ by
$\frac{1}{2}(I+\stochmat)$.  The resulting symmetric matrix has the
same eigenstructure as $\stochmat$, and moreover, we have
\begin{equation}
  \sigma_2 \Big(\half(I+P) \Big)
  = \lambda_2 \Big(\half(I+P) \Big)
  = \lambda_2 \Big(I - \frac{1}{2(\dmax + 1)}D^{1/2}\laplacian D^{1/2}
  \Big) \leq 1 - \frac{\dmin}{2(\dmax + 1)}\lambda_{n-1}(\laplacian).
  \label{eqn:lazy-walk-sigma2}
\end{equation}
Consequently, it is sufficient to bound only $\lambda_{n-1}(\laplacian)$,
which is more convenient from a technical standpoint. The convergence rate
implied by the lazy random walk through
Theorem~\ref{theorem:simple-convergence} is no worse than twice that
of the original walk, which is insignificant for the analysis in this paper.

%%%%%%%%%%%%%%%%%%%%%%%%%%%%%%%%%%%%%%%%%%%%%%%%%%%%%%%%%%%%%%%%%%%%%%%%

\vspace*{.2in}

\noindent We are now equipped to address each of the graph classes covered by
Corollary~\ref{corollary:graph-catalog}.

\paragraph{Cycles and paths:} Recall the regular $k$-connected cycle from
Figure~\ref{fig:graph-types}(a), constructed by placing the $n$ nodes
on a circle and connecting every node to $k$ neighbors on the right
and left.  For this graph, the Laplacian $\laplacian$ is a circulant
matrix with diagonal entries $1$ and off-diagonal non-zero entries
$-1/2k$.  Known results on circulant matrices (see Chapter 3 of
Gray~\cite{Gray06}) imply that it has $m$th eigenvalue
\begin{equation*}
  \lambda_m(\laplacian) = 1 - \frac{1}{2k} \sum_{j=1}^k \exp\left(-2\pi
  i j m / n\right) - \frac{1}{2k} \sum_{j=1}^k \exp\left(-2\pi i (n - j)
  m / n \right) = 1 - \frac{1}{k} \sum_{j=1}^k \cos\left(\frac{2\pi j
    m}{n}\right).
\end{equation*}
For $m = n - 1$ and $k = o(n)$, the last equation can be massaged
into~\cite[Section VI.A]{BoydGhPrSh06}
\begin{equation*}
  \lambda_{n-1}(\laplacian)
  = 1 - \cos\left(\frac{2\pi k}{n}\right) + \Theta\left(\frac{k^4}{n^4}\right).
\end{equation*}
By performing a Taylor expansion of $\cos(\cdot)$, we see that
$\lambda_{n-1}(\laplacian) = \Theta\left(\frac{k^2}{n^2}\right)$ for $k =
o(n)$.

Now consider the regular $k$-connected path, a path in which each node
is connected to the $k$ neighbors on its right and left.  By computing
Cheeger constants (see Lemma~\ref{proposition:k-connected-path} in
Appendix~\ref{sec:k-connected-path}), we see that if $k \le
\sqrt{n}$, then $\lambda_{n-1}(\laplacian) = \Theta(k^2/n^2)$.
Note also that for the $k$-connected path on $n$ nodes, $\min_i
\degree_i = k$ and $\dmax = 2k$. Thus, we can combine the previous
two paragraphs with Lemma~\ref{lemma:stochmat-laplacian} to see that for
regular $k$-connected paths or cycles with $k \le \sqrt{n}$,
\begin{equation}
  \label{eqn:eigenvalue-k-path}
  \sigma_2(\stochmat) = 1 - \Theta\left(\frac{k^2}{n^2}\right).
\end{equation}
Substituting the bound~\eqref{eqn:eigenvalue-k-path} into
Theorem~\ref{theorem:simple-convergence} yields the claim of
Corollary~\ref{corollary:graph-catalog}(a).

\paragraph{Regular grids:}  Now consider the case of a 
$\sqrt{n}$-by-$\sqrt{n}$ grid, focusing specifically on regular
$k$-connected grids, in which any node is joined to every node that is
fewer than $k$ horizontal or vertical edges away in an axis-aligned
direction. In this case, we use results on Cartesian products of
graphs~\cite[Section 2.6]{Chung98} to analyze the eigen-structure of
the Laplacian.  In particular, the toroidal $\sqrt{n}$-by-$\sqrt{n}$
$k$-connected grid is simply the Cartesian product of two regular
$k$-connected cycles of $\sqrt{n}$ nodes.  The second
smallest eigenvalue of a Cartesian product of graphs is half the
minimum of second-smallest eigenvalues of the original
graphs~\cite[Theorem 2.13]{Chung98}. Thus, based on the preceding
discussion of $k$-connected cycles, we conclude that if $k =
o(\sqrt{n})$, then we have $\lambda_{n-1}(\laplacian) =
\Theta(k^2/n)$.
For a non-toroidal $\sqrt{n}$-by-$\sqrt{n}$ grid (in which the network
is \emph{not} ``wrapped'' on its boundaries, as in
Figure~\ref{fig:graph-types}(b)), we use the previous discussion of
regular $k$-connected paths, since the grid is the Cartesian product
of two $k$-connected paths of $\sqrt{n}$ nodes. We immediately see that
$\lambda_{n-1}(\laplacian) = \Theta(k^2/n)$. In both cases, for
$\sqrt{n}$-by-$\sqrt{n}$ $k$-connected grids, we use
Lemma~\ref{lemma:stochmat-laplacian} and \eqref{eqn:lazy-walk-sigma2} to
see that for $k \le n^{1/4}$,
\begin{equation}
  \label{eqn:eigenvalue-k-grid}
  \sigma_2(\stochmat) = 1 - \Theta\left(\frac{k^2}{n}\right).
\end{equation}
The result in Corollary~\ref{corollary:graph-catalog}(b) immediately follows.

\paragraph{Random geometric graphs:}
 Using the proof of Lemma~10 from Boyd et al.~\cite{BoydGhPrSh06}, we
see that for any $\epsilon > 0$, if $r = \sqrt{\log^{1 + \epsilon} n /
(n\pi)}$, then with probability at least $1 - 2 / n^{c-1}$,
\begin{equation}
  \label{eqn:degree-geometric}
  \min_i \degree_i \ge \log^{1 + \epsilon} n - \sqrt{2} c \log n
  ~~~ {\rm and} ~~~
  \max_i \degree_i \le \log^{1 + \epsilon} n + \sqrt{2} c \log n.
\end{equation}
Thus, letting $\laplacian$ be the graph Laplacian of a random geometric graph,
if we can bound $\lambda_{n-1}(\laplacian)$, \eqref{eqn:degree-geometric}
coupled with Lemma~\ref{lemma:stochmat-laplacian} will control the convergence
rate of our algorithm.

Recent work of von Luxburg et al.~\cite{LuxburgRaHe10} gives
concentration results on the second-smallest eigenvalue of a geometric
graph. In particular, their Theorem 3 says that there are universal constants
$c_1, \ldots, c_5 > 0$ such that with probability at least $1 - c_1 n
\exp(-c_2 n r^2) - c_3 \exp(-c_4 n r^2) / r^2$, $\lambda_{n-1}(\laplacian) \ge
c_5 r^2$. Parsing this a bit, we see that if $r = \omega(\sqrt{\log n / n})$,
then with exceedingly high probability, $\lambda_{n-1}(\laplacian) = \Omega(r)
= \omega(\log n / n)$. Using \eqref{eqn:degree-geometric}, we see that for $r
= (\log^{1 + \epsilon} n / n)^{1/2}$,
\begin{equation*}
\frac{\min_i \degree_i}{\max_i \degree_i} = \Theta(1)
~~~ {\rm and} ~~~
\lambda_{n-1}(\laplacian) =
\Omega\left(\frac{\log^{1 + \epsilon}n}{n}\right)
\end{equation*}
with high probability. Combining the above equation with
Lemma~\ref{lemma:stochmat-laplacian} and \eqref{eqn:lazy-walk-sigma2}, we have
\begin{equation}
  \label{eqn:eigenvalue-geometric-graph}
  \sigma_2(\stochmat) = 1 - \Omega\left(\frac{\log^{1 + \epsilon}n}{n}
  \right).
\end{equation}
Thus we have obtained the result of
Corollary~\ref{corollary:graph-catalog}(c). Our bounds show that a
grid and a random geometric graph exhibit the same convergence rate up
to logarithmic factors.

\paragraph{Expanders:}
The constant spectral gap in expanders~\cite[Chapter 6]{Chung98}
removes any penalty due to network communication (up to logarithmic factors),
and hence yields Corollary~\ref{corollary:graph-catalog}(d).

%%%%%%%%%%%%%%%%%%%%%%%%%%%%%%%%%%%%%%%%%%%%%%%%%%%%%%%%%%%%%%%%%%%%%%%

\subsection{Proof of Proposition~\ref{proposition:tight-spectral-gap}}
\label{sec:lower-bound}

We now give a proof of
Proposition~\ref{proposition:tight-spectral-gap}, which shows that the
dependence of our convergence rates on the spectral gap is tight. The
proof is based on construction of a set of objective functions $f_i$
that force convergence to be slow by using the second eigenvector of
the communication matrix $\stochmat$.

Recall that $\onevec \in \real^n$ is the eigenvector of $\stochmat$
corresponding to its largest eigenvalue (equal to $1$). Let $v \in
\real^n$ be the eigenvector of $\stochmat$ corresponding to its second
singular value, $\sigma_2(\stochmat)$.  By using the lazy random walk
defined in Section~\ref{sec:communication-graph}, we may assume
without loss of generality that $\lambda_2(\stochmat) =
\sigma_2(\stochmat)$.  Let $w = \frac{v}{\linf{v}}$ be a normalized
version of the second eigenvector of $\stochmat$, and note that
$\sum_{i=1}^n w_i = 0$. Without loss of generality, we assume that
there is an index $i$ for which $w_i = -1$ (otherwise we can flip
signs in what follows); moreover, by re-indexing as needed, we may
assume that $w_1 = -1$.  We set $\xdomain = [-1, 1] \subset \R$, and
define the univariate functions $f_i(x) \defeq (c + w_i) x$, so that
the global problem is to minimize
\begin{equation*} 
\frac{1}{n}\sum_{i=1}^n f_i(x) = \frac{1}{n} \sum_{i=1}^n (c + w_i) x
= c x
\end{equation*}
for some constant $c > 0$ to be chosen.  Note that each $f_i$ is $c +
1$-Lipschitz.  By construction, we see immediately that $x^* = -1$ is
optimal for the global problem.

Now consider the evolution of the $\{z(t)\}_{t=0}^\infty \subset \R^n$,
as generated by the update~\eqref{eqn:lazy-unproject}. By
construction, we have $g_i(t) = c + w_i$ for all $t = 1,2,\ldots$.
Defining the vector $g = (c \onevec + w) \in \R^n$, we have the
evolution
\begin{align}
  z(t + 1) & = \stochmat z(t) + g = \stochmat^2 z(t - 1) + \stochmat g + g
  = \cdots =
  \sum_{\tau=0}^t \stochmat^\tau g \nonumber \\
  & = \sum_{\tau = 0}^{t-1} \stochmat^\tau (w + c \onevec)
  =  \sum_{\tau = 0}^{t-1} \stochmat^\tau w + c t \onevec
  = \sum_{\tau=0}^{t-1} \sigma_2(\stochmat)^\tau w + c t \onevec
  \label{eqn:z-lower-evolution}
\end{align}
since $\stochmat \onevec = \onevec$.

In order to establish a lower bound, it suffices to show that at least
one node is far from the optimum after $t$ steps, and we focus on node
$1$.  Since $w_1 = -1$, the evolution~\eqref{eqn:z-lower-evolution}
guarantees that
\begin{equation}
  \label{EqnZbehave}
  z_1(t + 1) = -\sum_{\tau=0}^{t-1} \sigma_2(\stochmat)^\tau + ct
  = ct - \frac{1 - \sigma_2(\stochmat)^{t-1}}{1 - \sigma_2(\stochmat)}.
\end{equation}
Recalling that $\prox(x) = \half x^2$ for this scalar setting, we have
\begin{equation*}
  x_i(t+1) = \argmin_{x \in \xdomain}\Big\{z_i(t+1) x +
  \frac{1}{2\stepsize(t)} x^2\Big\}
  = \argmin_{x \in \xdomain}
  \Big\{\big(x + \stepsize(t)z_i(t+1)\big)^2\Big\} 
\end{equation*}
Hence $x_1(t)$ is the projection of $-\stepsize(t)z_1(t+1)$ onto $[-1, 1]$,
and unless $z_1(t) > 0$ we have
\begin{equation*}
  f(x_1(t)) - f(-1) \geq c > 0.
\end{equation*}
If $t$ is overly small, the relation~\eqref{EqnZbehave} will
guarantee that $z_1(t) \le 0$, so that $x_1(t)$ is far from the
optimum.  If we choose $c \leq 1/3$, then a little calculation shows
that we require $t = \Omega((1 - \sigma_2(\stochmat))^{-1})$ in order
to drive $z_1(t)$ below zero.

\section{Convergence rates for stochastic communication}
\label{sec:random-P-convergence}

In this section, we develop theory appropriate for stochastic and
time-varying communication, which we model by a sequence
$\{\stochmat(t)\}_{t=0}^\infty$ of random matrices.  We begin in
Section~\ref{SecBasicStochastic} with basic convergence results in
this setting, and then prove
Theorem~\ref{theorem:random-P-convergence}.  Section~\ref{SecGossip}
contains a more detailed treatment of the case of gossip algorithms,
and Section~\ref{SecEdgeFail} contains the setting of edge
failures.

\subsection{Basic convergence analysis}
\label{SecBasicStochastic}

Recall that Theorem~\ref{theorem:master-convergence} involves the sum
$\frac{2L}{n} \sum_{t=1}^T \sum_{i=1}^n \stepsize(t)\dnorm{\bar{z}(t)
  - z_i(t)}$. In Section~\ref{sec:fixed-p-convergence}, we showed how
to control this sum when communication between agents occurs on a
static underlying network structure via a doubly-stochastic matrix
$\stochmat$.  We now relax the assumption that $\stochmat$ is fixed
and instead let $\stochmat(t)$ vary over time.

\subsubsection{Markov chain mixing for stochastic communication}

We use $\stochmat(t) = [\stochvec_1(t) ~ \cdots ~ \stochvec_n(t)]$ to
denote the doubly stochastic symmetric matrix at iteration $t$. The
update employed by the algorithm, modulo changes in $\stochmat$, is
given by the usual updates~\eqref{eqn:lazy-unproject}
and~\eqref{eqn:lazy-project}---namely,
\begin{equation*}
  z_i(t + 1) = \sum_{j=1}^n \stochvec_{ij}(t) z_j(t) + g_i(t), ~~~ x_i(t
  + 1) = \Pi_\xdomain^\prox(z_i(t + 1),\stepsize).
\end{equation*}
In this case, our analysis makes use of the modified definition $\Phi(t, s) =
\stochmat(s) \stochmat(s + 1) \cdots \stochmat(t)$. However, we still
have the evolution of $\bar{z}(t + 1) = \bar{z}(t) - \ninv
\sum_{i=1}^n g_i(t)$ from equation~\eqref{eqn:mean-z-update}, and
moreover, \eqref{eqn:bar-z-zi-diff} holds essentially unchanged:
\begin{equation}
  \label{eqn:bar-z-zi-diffNEW} \bar{z}(t) - z_i(t) = \sum_{s=1}^{t-1}
  \sum_{j=1}^n (1/n - [\Phi(t - 1, s)]_{ji}) g_j(s - 1) +
  \frac{1}{n} \sum_{j=1}^n \left(g_j(t-1) - g_i(t - 1)\right).
\end{equation}
To show convergence for the random communication model, we must
control the convergence of $\Phi(t, s)$ to the uniform distribution.
We first claim that
\begin{equation}
  \label{eqn:probability-phi}
  \P \big[\ltwo{\Phi(t, s) e_i - \onevec/n} \ge \epsilon \big] \le
  \epsilon^{-2} \lambda_2 \left(\E [P(t)^2]\right)^{t - s + 1},
\end{equation}
which we establish by recalling and modifying a few known
results~\cite{BoydGhPrSh06}.

Let $\Delta_n$ denote the $n$-dimensional probability simplex and
the vector $u(0) \in \Delta_n$ be arbitrary. Consider the random
sequence $\{u(t)\}_{t=0}^\infty$ generated by the recursion $u(t + 1)
= \stochmat(t) u(t)$.  Let $v(t) \defeq u(t) - \onevec/n$
correspond to the portion of $u(t)$ orthogonal to the all $1$s
vector.  Calculating the second moment of $v(t + 1)$, we have
\begin{align*}
  \E \big [\<v(t+1), v(t + 1)\> \mid v(t) \big] & = \E \big [v(t)^T
    \stochmat(t)^T \stochmat(t) v(t) \mid v(t) \big] = v(t)^T
  \E[\stochmat(t)^T \stochmat(t) ] v(t) \\ & \le \ltwo{v(t)}^2
  \lambda_2 \big( \E \stochmat(t)^T \stochmat(t) \big) = \ltwo{v(t)}^2
  \lambda_2(\E \stochmat(t)^2)
\end{align*}
since $\<v(t), \onevec\> = 0$, $v(t)$ is orthogonal to the first
eigenvector of $\stochmat(t)$, and $\stochmat(t)$ is symmetric.
Applying Chebyshev's inequality yields
\begin{align*}
  \P \Bigg[\frac{\|u(t) - \onevec/n \|_2}{\|u(0)\|_2} \geq \epsilon
  \Bigg] & \leq \frac{\E \|v(t)\|^2}{\|u(0)\|_2^2 \; \epsilon^2} \\
  & \leq \epsilon^{-2} \frac{\ltwo{v(0)}^2 \lambda_2 \big(\E
    \stochmat(t)^2\big)^t}{\ltwo{u(0)}^2}.
\end{align*}
Replacing $u(0)$ with $e_i$ and noting that $\|e_i - \onevec/n\|_2
\le 1$ yields the claim~\eqref{eqn:probability-phi}.

\subsubsection{Proof of Theorem~\ref{theorem:random-P-convergence}}
  
Using the claim~\eqref{eqn:probability-phi}, we now prove the main
theorem of this section, following an argument similar to the proof of
Theorem~\ref{theorem:simple-convergence}.  We begin by choosing a
(non-random) time index $\what{t}$ such that for $t - s \ge \what{t}$,
with exceedingly high probability, $\Phi(t, s)$ is close to the
uniform matrix $\onevec \onevec^T/n$.  We then break the summation
from $1$ to $T$ into two separate terms, separated by the cut-off
point $\what{t}$.  Throughout this derivation, we let $\lambda_2 =
\lambda_2(\E [\stochmat(t)^2])$, where we have suppressed the
dependence of $\lambda_2$ on graph structure $\graph$ to ease
notation.

Using the probabilistic bound~\eqref{eqn:probability-phi}, note that
\begin{equation*}
  t - s \ge \frac{3\log \epsilon^{-1}}{\log \lambda_2^{-1}} - 1
  \quad {\rm implies} \quad
  \P\big[\ltwo{\Phi(t, s) e_i - \onevec/n} \geq
  \epsilon \big] \leq \epsilon.
\end{equation*}
Consequently, if we make the choice
\begin{equation*}
  % \label{EqnTchoice}
  \what{t} \; \defeq \; \frac{3 \log (T^2 n)}{\log \lambda_2^{-1}} =
  \frac{6 \log T + 3\log n}{\log \lambda_2^{-1}} \leq \frac{6 \log T +
    3\log n}{1 - \lambda_2},
\end{equation*}
then we are guaranteed that if $t - s \ge \what{t} - 1$, then
\begin{align}
  \label{eqn:probability-phi-tn2}
  \P \big[\ltwo{\Phi(t, s) e_i - \onevec/n} \ge 1/(T^2n) \big] \leq
  (T^2 n)^2 \lambda_2^{\frac{3 \log (T^2 n)}{ -\log \lambda_2}} \; = \;
  (T^2n)^2 \big(e^{\log \lambda_2}\big)^{ \frac{\log (T^6n^3)}{-\log
      \lambda_2}} \; = \; \frac{1}{T^2n}.
\end{align}
Recalling the bound~\eqref{eqn:intermediate-bar-z-zi-bound}, we have
\begin{align}
  \dnorm{\bar{z}(t) - z_i(t)} & \le L \sum_{s=1}^{t-1}\lone{\Phi(t - 1,
    s) e_i - \onevec/n} + 2L \nonumber \\ & = L \sum_{s = t -
    \what{t}}^{t - 1} \lone{\Phi(t - 1, s) e_i - \onevec/n} + L
  \sum_{s = 1}^{t - 1 - \what{t}} \lone{\Phi(t - 1, s) e_i -
    \onevec/n} + 2L \nonumber \\ & \le 2L \frac{3 \log (T^2 n)}{1 -
    \lambda_2} + \underbrace{L \sqrt{n} \sum_{s=1}^{t - 1 - \what{t}}
    \ltwo{\Phi(t - 1, s) e_i - \onevec/n}}_{\term} + 2L.
  \label{eqn:diff-z-to-bound-with-prob}
\end{align}

It remains to bound the sum $\term$.  For any fixed pair $s' < s$,
since the matrices $\stochmat(t)$ are doubly stochastic, we have
\begin{align*}
  \ltwo{\Phi(t - 1, s') e_i - \onevec/n} & = \ltwo{\Phi(s-1, s')
    \Phi(t - 1, s) e_i - \onevec/n} \\
  & \leq \matrixnorm{\Phi(s-1, s')}_2 \ltwo{\Phi(t-1, s) e_i -
    \onevec/n} \\
  & \leq \ltwo{\Phi(t-1, s) e_i - \onevec/n},
\end{align*}
where the final inequality uses the bound
$\matrixnorm{\Phi(s-1, s')}_2 \le 1$.
From the bound~\eqref{eqn:probability-phi-tn2}, we have the bound
$\ltwo{\Phi(t - 1, t - \what{t} - 1) e_i - \onevec/n} \le
\frac{1}{T^2n}$ with probability at least $1 - 1/(T^2n)$.  Since $s$
ranges between $1$ and $t-\what{t}$ in the summation $\term$, we conclude
that
\begin{align*}
\term & \leq L \sqrt{n} \, T \; \frac{1}{T^2 n} \; = \frac{L \sqrt{n}}{T^2
  \numnode},
\end{align*}
and hence assuming that $n \ge 3$,
\begin{equation*}
  \dnorm{\bar{z}(t) - z_i(t)} \le L \frac{6 \log (T^2n)}{1 - \lambda_2}
  + L \sqrt{n} \frac{1}{Tn} + 2 L
\end{equation*}
with probability at least $1 - 1/(T^2n)$.
Applying the union bound over all iterations $t = 1, \ldots, T$ and
nodes $i = 1, \ldots, \numnode$, we obtain 
\begin{align*}
  \P \bigg[\max_{t \le T} \max_{i \le n} \dnorm{\bar{z}(t) - z_i(t)} >
    \frac{6 L \log (T^2n)}{1 - \lambda_2} + \frac{L}{T\sqrt{n}} + 2
    L \bigg]
  & \leq \frac{1}{T}.
\end{align*}
Recalling the master bound from
Theorem~\ref{theorem:master-convergence} completes the proof.

\vspace*{.1in}

\noindent
In the remainder of this section, we give some applications of the
stochastic framework outlined above, showing a few sampling schemes
and giving bounds on their convergence rates.

%%%%%%%%%%%%%%%%%%%%%%%%%%%%%%%%%%%%%%%%%%%%%%%%%%%%%%%%%%%%%%%%%%%%%%%%%%%%

\subsection{Gossip-like protocols}
\label{SecGossip}

Gossip algorithms are procedures for achieving consensus in a network
robustly and quickly by randomly selecting one edge $(i, j)$ in the
network for communication at each iteration~\cite{BoydGhPrSh06}.  Once
nodes $i$ and $j$ are selected, their values are averaged. Gossip
algorithms drastically reduce communication in the network, yet they
still enjoy fast convergence and are robust to changes in topology.

%% In this subsection, we analyze partially and
%% completely asynchronous gossip-like communication protocols.

\subsubsection{Partially asynchronous gossip protocols}

In a partially asynchronous iterative method, agents synchronize their
iterations~\cite{BertsekasTs89}. This is the model of standard gossip
protocols, where computation proceeds in rounds, and in each round
communication occurs on one random edge.  In our framework, this
corresponds to using the random transition matrix $\stochmat(t) = I -
\half(e_i - e_j)(e_i-e_j)^T$. It is clear that
$\stochmat(t)^T\stochmat(t) = \stochmat(t)$, since $\stochmat(t)$ is a
projection matrix.

%% This is an attractive analytical property as it guarantees that
%% $\lambda_2(\E\stochmat(t)^2) = \lambda_2(\E\stochmat(t))$. Computing the
%% expectation of the random transition matrix can be much simpler in practice
%% than that of the squared transition matrix. Clearly, this property holds
%% whenever $\stochmat(t)$ is a projection matrix. Another protocol for which
%% this property always holds is to base $\stochmat(t)$ on a random matching
%% of nodes chosen at every round (a matching is a vertex-disjoint collection
%% of edges). In the remainder of this section, we relate the spectrum of the
%% expected transition matrix to the underlying graph.

%% To analyze the gossip algorithm, we assume that at each round of computation,
%% one edge is chosen uniformly at random from the graph and the nodes incident
%% on that edge average their vectors.

Let $A$ be the adjacency matrix of the graph $\graph$ and $D$ be the
diagonal matrix of its degrees as in
Section~\ref{sec:communication-graph}.  At round $t$, edge $(i, j)$
(with $A_{ij} = 1$) is chosen with probability $1/ \<\onevec, A
\onevec\>$. Thus,
\begin{align}
  \E \stochmat(t) & = \frac{1}{\<\onevec, A\onevec\>}
  \sum_{(i,j) : A_{ij} = 1} I - \half(e_i - e_j)(e_i - e_j)^T
  = I - \frac{1}{\<\onevec, A\onevec\>}(D - A) \nonumber \\
  & = I - \frac{1}{\<\onevec, A\onevec\>} D^{1/2}(I - D^{-1/2} A D^{-1/2})D^{1/2}
  = I - \frac{1}{\<\onevec, A\onevec\>} D^{1/2} \laplacian D^{1/2}
  \label{eqn:expectation-gossip}
\end{align}
since $\sum_{(i,j) : A_{ij} = 1} (e_i - e_j)(e_i - e_j)^T = 2(D - A)$. Using
an identical argument as that for Lemma~\ref{lemma:stochmat-laplacian}, we see
that \eqref{eqn:expectation-gossip} implies that
\begin{equation*}
\lambda_2(\E \stochmat(t))
\le 1 - \frac{\min_i \degree_i}{\<\onevec, A\onevec\>}
\lambda_{n-1}(\laplacian).
\end{equation*}
Note that $\<\onevec, A \onevec\> = \<\onevec, D \onevec\>$, so that for
approximately regular graphs, $\<\onevec, A \onevec\> \approx n \dmax$, and
$\min_i \degree_i / \<\onevec, A \onevec\> \approx 1 / n$. Thus, at the
expense of a factor of roughly $1/n$ in convergence rate, we can reduce the
number of messages sent per round from the number of edges in the graph,
$\Theta(n \dmax)$, to one. In a clustered computing environment with some
centralized control, it is possible to select more than one edge per round so
long as no two edges share vertices (for example, by selecting a random
maximal matching) and still have $\stochmat(t)^T \stochmat(t) =
\stochmat(t)$. For a $\degree$-regular graph, choosing a random maximal
matching achieves a spectral gap within constant factors of the spectral gap
of the underlying graph but uses only $\Theta(1/\degree)$ as much
communication.

\subsubsection{Totally asynchronous gossip protocol}

Now we relax the assumption that agents have synchronized clocks, so
the iterations of the algorithm are no longer synchronized. Suppose
that each agent has a random clock ticking at real-valued times, and
at each clock tick, the agent randomly chooses one of its neighbors to
communicate with. Further assume that each agent computes an iterative
approximation to $g_i \in \partial f_i(x_i(t))$, and that the
approximation is always unbiased (an example of this is when $f_i$ is
the sum of several functions, and agent $i$ simply computes the
subgradient of each function sequentially). We assume that no two
agents have clocks tick at the same time.  This communication
corresponds to a gossip protocol with stochastic subgradients, and its
convergence can be described simply by combining
\eqref{eqn:expectation-gossip} with
Theorem~\ref{theorem:stochastic-gradient}. This type of algorithm is
well-suited to completely decentralized environments, such as sensor
networks.

\subsection{Random edge inclusion and failure}
\label{SecEdgeFail}

The two communication ``protocols'' we analyze now make selection of
each edge at each iteration of the algorithm independent. We begin
with random edge inclusions and follow by giving convergence
guarantees for random edge failures. For both protocols, since
computation of $\E \stochmat(t)^2$ is in general non-trivial, we work
with the model of lazy random walks described in
Section~\ref{sec:communication-graph}. In the lazy random walk model, the
communication matrix at each round is $\half I + \half \stochmat(t)$,
which is symmetric PSD since $\sigma_1(\stochmat(t)) \le 1$. Further,
for any symmetric PSD stochastic matrix $\stochmat$, $\stochmat^2
\preceq \stochmat$.  With that in mind, we see that $\E(\half I +
\half \stochmat(t))^2 \preceq \half I + \half \E \stochmat(t)$, and
applying Weyl's Theorem for the eigenvalues of a Hermitian
matrix~\cite[Theorem 4.3.1]{HornJo85},
\begin{equation}
  \lambda_2\bigg(\E \Big(\half I + \half \stochmat(t)\Big)^2\bigg)
  \le \lambda_2\bigg(\half I + \half \E \stochmat(t)\bigg)
  = \half + \half \lambda_2(\E \stochmat(t)).
  \label{eqn:weyls-stochmatrices}
\end{equation}
Thus any bound on $\lambda_2(\E \stochmat(t))$ provides an upper bound
on the convergence rate of the distributed dual averaging algorithm
with random communication, as in
Theorem~\ref{theorem:random-P-convergence}.

Consider the communication protocol in which with probability $1 -
\degree_i / (\dmax + 1)$, node $i$ does not communicate, and otherwise
the node picks a random neighbor. If a node $i$ picks a neighbor $j$,
then $j$ also communicates back with $i$ to ensure double
stochasticity of the transition matrix. We let $A(t)$ be the random
adjacency matrix at time $t$.  When there is an edge $(i, j)$ in the
underlying graph, the probability that node $i$ picks edge $(i, j)$ is
$1/(\dmax + 1)$, and thus $\E A(t)_{ij} =
\frac{2\dmax+1}{(\dmax+1)^2}$. The random communication matrix is
$\stochmat(t) = I - (\dmax+1)^{-1}(D(t) - A(t))$. Let $A$ and $D$ be
the adjacency matrix and degree matrix of the underlying
(non-stochastic) graph and $\stochmat$ be communication matrix defined
in \eqref{eqn:mod-stochmat}. With these definitions, $\E
A(t) = \frac{2\dmax + 1}{(\dmax + 1)^2} A$, $\E D(t) = \frac{2 \dmax +
1}{(\dmax + 1)^2} D$, and $A - D = (\stochmat - I)(\dmax + 1)$. We
have
\begin{equation*}
  \E\stochmat(t) = I - (\dmax+1)^{-1}(\E D(t) - \E A(t))
  %% = I - \frac{1}{\dmax + 1}\left[\frac{2 \dmax + 1}{(\dmax + 1)^2}
  %%   D - \frac{2 \dmax + 1}{(\dmax + 1)^2} A\right]
  = \left(\frac{\dmax}{\dmax+1}\right)^2I +
  \frac{2\dmax+1}{(\dmax+1)^2}\stochmat,
\end{equation*}
and hence
\begin{equation*}
  1-\lambda_2(\E\stochmat(t)) =
  \frac{2\dmax+1}{(\dmax+1)^2}(1-\lambda_2(\stochmat)).
\end{equation*}
Using \eqref{eqn:weyls-stochmatrices}, we see that the spectral gap decreases
(and hence convergence rate may slow) by a factor proportional to the maximum
degree in the graph. This is not surprising, since the amount of communication
performed decreases by the same factor.

A related model we can analyze is that of a network in which at every
time step of the algorithm, an edge fails with probability $\failprob$
independently of the other edges. We assume we are using the model of
communication in the prequel, so $\stochmat(t) = I - (\dmax + 1)^{-1}
D(t) + (\dmax + 1)^{-1} A(t)$.  Let $A$, $D$, and $\stochmat$ be as
before and $\laplacian$ be the Laplacian of the underlying graph; we
easily have
\begin{equation*}
\E \stochmat(t) = I - \frac{1 - \failprob}{\dmax + 1} D - \frac{1 -
\failprob}{\dmax + 1} A = I - \frac{1 - \failprob}{\dmax + 1} D^{1/2}
\laplacian D^{1/2} = \failprob I + (1 - \failprob) \stochmat
\end{equation*}
and $\lambda_2(\E \stochmat(t)) = \rho + (1 -
\rho)\lambda_2(\stochmat)$. Applying
\eqref{eqn:simplified-stochmat-convergence}, we see that we lose at
most a factor of $\sqrt{1-\rho}$ in the convergence rate.

%%%%%%%%%%%%%%%%%%%%%%%%%%%%%%%%%%%%%%%%%%%%%%%%%%%%%%%%%%%%%%%%%%%%%%%%%%%

\section{Stochastic Gradient Optimization}
\label{sec:stoch-gradient}

In this section, we show that the algorithm we have presented naturally
generalizes to the case in which the agents do not receive true subgradient
information but only an unbiased estimate of a subgradient of $f_i$. That is,
during round $t$ agent $i$ receives a vector $\what{g}_i(t)$ with $\E
\what{g}_i(t) = g_i(t) \in \partial f_i(x_i(t))$. The proof is made
significantly easier by the dual averaging algorithm, which by virtue
of the simplicity of its dual update smooths the propagation of errors
from noisy estimates of individual 
subgradients throughout the network. This was a difficulty in
prior work, where significant care was needed in the analysis to address
passing noisy gradients through nonlinear projections~\cite{RamNeVe10}.

%% For example, if $f_i$ is defined via a probability measure $\mu$ as
%% $f_i(x) = \int f_i(x; \omega) d\mu(\omega)$ and $f_i(x; \omega)$ is
%% $L$-Lipschitz for all $\omega$, then randomly sampling $\omega$ from the
%% measure $\mu$ and using $g_i(t) = f_i'(x_i(t); \omega)$ satisfies these
%% constraints~\cite[e.g.][]{RockafellarWe82}.

\subsection{Proof of Theorem~\ref{theorem:stochastic-gradient}}
We begin by using convexity and the Lipschitz continuity of the $f_i$
(see equations~\eqref{eqn:dist-lazy-to-bound}
and~\eqref{eqn:first-order-dist-bound}), thereby obtaining
that the running sum $S(T) = \sum_{t=1}^T f(y(t)) - f(x^*)$ is upper bounded
as
\begin{align} 
  S(T) & \le \sum_{t=1}^T \ninv \sum_{i=1}^n \<g_i(t), x_i(t) - x^*\> +
  \sum_{t=1}^T \sum_{i=1}^n L \norm{y(t) - x_i(t)} \nonumber \\
  & = \sum_{t=1}^T \ninv \sum_{i=1}^n \<\what{g}_i(t), x_i(t) - x^*\>
  + \sum_{t=1}^T \ninv \sum_{i=1}^n L \norm{y(t) - x_i(t)} +
  \sum_{t=1}^T \ninv \sum_{i=1}^n \<g_i(t) - \what{g}_i(t), x_i(t) -
  x^*\>.
  \label{eqn:stochastic-to-bound}
\end{align}
We bound the first two terms of \eqref{eqn:stochastic-to-bound} using the
same derivation as that for
Theorem~\ref{theorem:master-convergence}. In particular, $\sum_{i=1}^n
\<\what{g}_i(t), x_i(t) - x^*\> = \sum_{i=1}^n \<\what{g}_i(t), y(t) - x^*\>
+ \sum_{i=1}^n \<\what{g}_i(t), x_i(t) - y(t)\>$, and nothing in
Lemma~\ref{lemma:ftrl-linear} assumes that $\what{g}_i(t)$ is related to
$f_i(x_i(t))$. So we upper bound the first term in
\eqref{eqn:stochastic-to-bound} with
\begin{align}
  \frac{1}{\stepsize(T)} \prox(x^*)
  + \half \sum_{t=1}^T \stepsize(t-1)\dnormb{\ninv
    \sum_{i=1}^n \what{g}_i(t)}^2
  + \sum_{t=1}^T \ninv \sum_{i=1}^n \<\what{g}_i(t), x_i(t) - y(t)\>.
  \label{eqn:intermediate-stochastic-bound}
\end{align}
H\"older's inequality implies that
$\E[\dnorm{\what{g}_i(t)}\dnorm{\what{g}_j(s)}] \le L^2$ and
$\E\dnorm{\what{g}_i(t)} \le L$ for any $i, j, s, t$. We use the two
inequalities to bound \eqref{eqn:intermediate-stochastic-bound}.
We have
\begin{equation*}
  \E \dnormb{\ninv\sum_{i=1}^n \what{g}_i(t)}^2
  \le \frac{1}{n^2}\sum_{i,j=1}^n \E
  \left[\dnorm{\what{g}_i(t)}\dnorm{\what{g}_j(t)}\right]
  \le L^2.
\end{equation*}
Further, $x_i(t) \in \mc{F}_{t-1}$ and $y(t) \in \mc{F}_{t-1}$ 
by assumption for $j \in [n]$ and $s \le t - 1$,
so
\begin{equation*}
  \E \<\what{g}_i(t), x_i(t) - y(t)\>
  \le \E \dnorm{\what{g}_i(t)}\norm{x_i(t) - y(t)}
  = \E \left(\E[\dnorm{\what{g}_i(t)} \mid \mc{F}_{t-1}] \norm{x_i(t) - y(t)}
  \right)
  \le L \E \norm{x_i(t) - y(t)}.
\end{equation*}
Recalling that $\norm{x_i(t) - y(t)} \le \stepsize(t) \dnorm{\bar{z}(t) -
  z_i(t)}$, we proceed by putting expectations around the norm terms in
\eqref{eqn:intermediate-bar-z-zi-bound} and \eqref{eqn:fixed-p-convergence}
to see that
\begin{equation*}
  \frac{1}{\stepsize(t)}\E \norm{y(t) - x_i(t)}
  \le \E \dnorm{\bar{z}(t) - z_i(t)}
  \le \sum_{s=1}^{t-1} L \lone{[\Phi(t - 1, s)]_i - \onevec/n}
  + 2 L
  \le L \frac{\log(T \sqrt{n})}{1 - \sigma_2(\stochmat)}
  + 3 L.
\end{equation*}

Coupled with the above arguments, we can bound the expectation of
\eqref{eqn:stochastic-to-bound} by
\begin{align}
  \E \left[\sum_{t=1}^T f(y(t)) - f(x^*) \right]
  & \le \frac{1}{\stepsize(T)} \prox(x^*) + \frac{L^2}{2}\sum_{t=1}^T
  \stepsize(t-1)
  + \left(2 L^2 \frac{\log(T \sqrt{n})}{1 - \sigma_2(\stochmat)}
  + 6 L^2\right)\sum_{t=1}^T \stepsize(t) \nonumber \\
  & ~~~~ + \ninv\sum_{t=1}^T \sum_{i=1}^n
  \E\left[\<g_i(t) - \what{g}_i(t), x_i(t) - x^*\>\right].
  \label{eqn:stochastic-martingale-to-bound}
\end{align}
Taking the expectation for the final term in the
bound~\eqref{eqn:stochastic-martingale-to-bound}, we recall
that $x_i(t) \in \mc{F}_{t-1}$, so
\begin{align}
  \E\left[\<g_i(t) - \what{g}_i(t), x_i(t) - x^*\> \right] & = \E
  \left[\E \left[ \<g_i(t) - \what{g}_i(t), x_i(t) - x^*\> \mid
      \mc{F}_{t-1} \right]\right] \nonumber \\ & = \E \left[
    \<\E(g_i(t) - \what{g}_i(t) \mid \mc{F}_{t-1}), x_i(t) - x^*\>
    \right] = 0,
  \label{eqn:expectation-stochgrad-zero}
\end{align}
which completes the proof of the first statement of the theorem.

To show that the statement holds with high-probability when $\xdomain$
is compact and $\dnorm{\what{g}_i(t)} \le L$, it is sufficient to
establish that the sequence $\langle g_i(t) - \what{g}_i(t), x_i(t) -
x^*\rangle$ is a bounded martingale, and then apply Azuma's
inequality~\cite{Azuma67}.  (Here we are exploiting the fact that
under compactness and bounded norm conditions, our previous bounds on
terms in the 
decomposition~\eqref{eqn:intermediate-stochastic-bound} now hold for
the analogous terms in the
decomposition~\eqref{eqn:stochastic-martingale-to-bound} without
taking expectations.)

By assumption on the compactness of $\xdomain$ and the
Lipschitz assumptions on $f_i$, we have
\begin{equation*}
  \<g_i(t) - \what{g}_i(t), x_i(t) - x^*\> \le \dnorm{g_i(t) -
    \what{g}_i(t)} \norm{x_i(t) - x^*} \le 2 L \radius.
\end{equation*}
Recalling \eqref{eqn:expectation-stochgrad-zero},
we conclude that the last sum in the
decomposition~\eqref{eqn:stochastic-martingale-to-bound} is a bounded
difference martingale, and Azuma's inequality implies that
\begin{equation*}
  \P \bigg[\sum_{t=1}^T \sum_{i=1}^n \<g_i(t) - \what{g}_i(t), x_i(t) -
  x^*\> \ge \epsilon \bigg] \leq \exp \big(-\frac{\epsilon^2}{16 T n^2
  L^2 \radius^2} \big).
\end{equation*}
Dividing by $T$ and setting the probability above equal to $\delta$,
we obtain that with probability at least $1 - \delta$,
\begin{align*}
  \frac{1}{nT}\sum_{t=1}^T \sum_{i=1}^n \<g_i(t) - \what{g}_i(t), x_i(t)
  - x^*\> \le 4L\radius \sqrt{\frac{\log \frac{1}{\delta}}{T}}.
\end{align*}
The second statement of the theorem is now obtained by appealing to
Lemma~\ref{lemma:transfer-y-to-x}.  By convexity, we have
$f(\what{x}_i(T)) \le \frac{1}{T} \sum_{t=1}^T f(x_i(t))$, thereby
completing the proof.

Proving the last statement of the theorem---the concentration result with
uncorrelated noise at each node---requires a martingale extension of
Bernstein's inequality~\cite{Freedman75}. Indeed, one form of Freedman's
inequality states that if $X_1, \ldots, X_T$ is a martingale difference
sequence, $|X_i| \le B$ uniformly, and $V \ge \sum_{t=1}^T \Var(X_t
\mid \mc{F}_{t-1})$, then for any $v, \epsilon > 0$,
\begin{equation*}
  \P\left(\sum_{t=1}^T X_t \ge \epsilon
  ~~ {\rm and} ~~
  V \le v\right) \le \exp\left(-\frac{\epsilon^2}{2 v + 2B  \epsilon / 3}
  \right).
\end{equation*}
To extend the above bound to our setting, we recall that
$\ninv \sum_{i=1}^n \<g_i(t) - \what{g}_i(t), x_i(t) - x^*\>$ is Martingale
difference sequence uniformly bounded by $2 L\radius$. Further, since
the expectation is zero, we have
\begin{align}
  \lefteqn{\Var\bigg(\ninv \sum_{i=1}^n \<g_i(t) - \what{g}_i(t),
    x_i(t) - x^*\> \mid \mc{F}_{t-1} \bigg)} \nonumber \\
  & = \E \bigg[\frac{1}{n^2} \sum_{i, j}^n
    \<g_i(t) - \what{g}_i(t), x_i(t) - x^*\>\<g_j(t) - \what{g}_j(t),
    x_j(t) - x^*\> \mid \mc{F}_{t-1} \bigg]
  \nonumber \\
  & = \frac{1}{n^2} \E \bigg[\sum_{i=1}^n \<g_i(t) - \what{g}_i(t),
    x_i(t) - x^*\>^2 \mid \mc{F}_{t-1} \bigg]
  \label{eqn:decorrelate-martingale-almost} \\
  & \le \frac{1}{n^2} \sum_{i=1}^n 4 L^2 \radius^2 = \frac{4L^2\radius^2}{n}.
  \nonumber % \label{eqn:decorrelate-martingale}
\end{align}
The decorrelation equality in \eqref{eqn:decorrelate-martingale-almost}
follows by our assumption that $\what{g}_i(t)$ and $\what{g}_j(t)$ are
uncorrelated given $\mc{F}_{t-1}$, and that $x_i(t)$, $g_i(t)$, and $x^* \in
\mc{F}_{t-1}$. Substituting $4 T L^2 \radius^2 / n$ as an upper bound
for the variance in Freedman's inequality, we have
\begin{equation*}
  \P \left(\ninv \sum_{t=1}^T \sum_{i=1}^n
  \<g_i(t) - \what{g}_i(t), x_i(t) - x^*\> \ge \epsilon\right)
  \le \exp\left(-\frac{\epsilon^2}{8 T L^2 \radius^2 / n +
    8 L\radius \epsilon / 3}\right).
\end{equation*}
To find a $\delta$ so that $\exp(\cdot)$ term is less than or equal to
$\delta$, we solve
\begin{equation}
  \delta \ge \exp\left(-\frac{\epsilon^2}{8 L^2\radius^2 / n +
    8 L\radius \epsilon/3} \right)
  ~~~ {\rm or} ~~~
  \epsilon^2 - \epsilon \frac{8 L \radius\log \frac{1}{\delta}}{3}
  - \frac{8 T L^2 \radius^2 \log \frac{1}{\delta}}{n}
  \ge 0.
  \label{eqn:decorrelate-solve-probability}
\end{equation}
Solving the above quadratic in $\epsilon$, we have equality with zero for
\begin{equation*}
  \epsilon = \frac{(8/3) L\radius \log\frac{1}{\delta} \pm
    \sqrt{(8/3)^2 L^2 \radius^2 \log^2 \frac{1}{\delta}
      + (32/n) T L^2 \radius^2  \log \frac{1}{\delta}}}{2}
\end{equation*}
In particular, noting that $\sqrt{a + b} \le \sqrt{a} + \sqrt{b}$, it is
sufficient that
\begin{equation*}
  \epsilon \ge \frac{4}{3} L\radius \log \frac{1}{\delta}
  + \frac{\sqrt{6}}{3} L\radius \log \frac{1}{\delta}
  + 2 \sqrt{2} L \radius \sqrt{\frac{T}{n} \log \frac{1}{\delta}}
\end{equation*}
for the inequality in \eqref{eqn:decorrelate-solve-probability} to be
satisfied. Thus with probability at least $1 - \delta$,
\begin{equation*}
  \ninv \sum_{t=1}^T \sum_{i=1}^n
  \<g_i(t) - \what{g}_i(t), x_i(t) - x^*\>
  < 3 L\radius \log \frac{1}{\delta}
  + 4 L\radius \sqrt{\frac{T}{n} \log \frac{1}{\delta}}.
\end{equation*}
Dividing by $T$ completes the proof of the last statement of
Theorem~\ref{theorem:stochastic-gradient}.

\section{Simulations}
\label{sec:simulations}

In this section, we report experimental results on the network scaling
behavior of the distributed dual averaging algorithm as a function of
the graph structure and number of processors $n$.  These results
illustrate the excellent agreement of the empirical behavior with our
theoretical predictions.

\begin{figure}[h]
  \begin{center}
    \includegraphics[width=0.45\columnwidth]{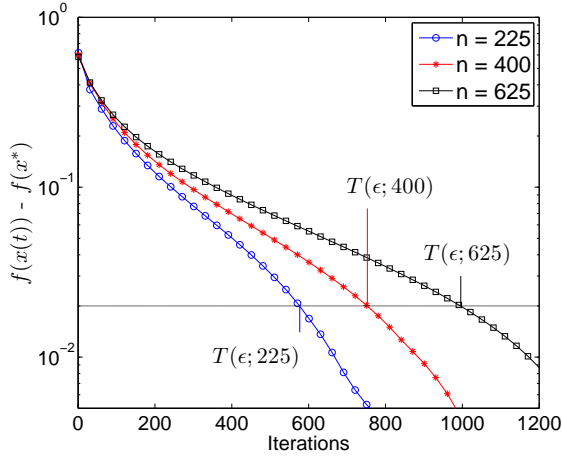}
    \caption{\label{fig:convergence} Plot of the function error versus
the number of iterations for a grid graph.  Each curve corresponds to
a grid with a different number of nodes ($n \in \{225, 400, 600 \}$.
As expected, larger graphs require more iterations to reach a
pre-specified tolerance $\epsilon > 0$, as defined by the iteration
number $T(\epsilon; n)$.  The network scaling problem is to determine
how $T(\epsilon; n)$ scales as a function of $n$.
 }
  \end{center}
\end{figure}

For all experiments reported here, we consider distributed
minimization of a sum of hinge loss functions; it is this optimization
problem that underlies the widely-used support vector machine method
for classification~\cite{CortesVa95}. In a classification problem, we
are given $n$ pairs of the form $(b_i, y_i) \in \real^d \times \{-1,
+1\}$, where $b_i \in \real^d$ corresponds to a feature vector and
$y_i \in \{-1, +1\}$ is the associated label.  The goal is to use
these samples to estimate a linear classifier, meaning a function of
the form $b \mapsto \sign \<b,x\>$ based on some weight vector $x \in
\real^d$.  In methods based on support vector machines, the weight
vector is chosen by minimizing a sum of hinge loss functions
associated with each pair $(b_i, y_i)$.  In particular, given the
shorthand notation $\hinge{c} \defeq \max\{0, c\}$, the hinge loss
associated with a linear classifier based on $x$ is given by $f_i(x) =
\hinge{1 - y_i \, \<b_i, x\>}$.  The global objective is a sum of $n$
such terms, namely
\begin{equation}
\label{EqnSumHinge}
  f(x) \; \defeq \; \frac{1}{n} \sum_{i=1}^n \hinge{1 - y_i \<b_i, x\>}.
\end{equation}
Setting $L = \max_i \|b_i\|_2$, we note that $f$ is
$L$-Lipschitz and non-smooth at any point with $\<b_i, x\> =
y_i$.  It is common to impose some type of quadratic constraint on the
minimization problem~\eqref{EqnSumHinge}, and for the simulations
considered here, we set $\xdomain = \{x \in \real^d \, \mid \,
\ltwo{x} \le 5\}$.  For a given graph size $n$, we form a random
instance of a SVM classification problem as follows.  For each $i = 1,
2, \ldots, n$, we first draw a random vector $b_i \in \R^d$ from the
uniform distribution over the unit sphere.  We then randomly generate
a random Gaussian vector $w \sim N(0, I_{d \times d})$, and then let
$a_i = \sign(\<w, b_i\>) b_i$, randomly flipping the sign of 5\% of
the $a_i$.  Note that these choices yield a function $f$ that is
Lipschitz with parameter $L = 1$. Although this is a specific ensemble
of problems, we have observed qualitatively similar behavior for other
problem ensembles.  In order to study the effect of graph size and
topology, we perform simulations with three different graph
structures, namely cycles, grids, and random $5$-regular
expanders~\cite{FriedmanKaSz89}, with the number of nodes $n$ ranging
from $100$ to $900$.  In all cases, we use the optimal setting of the
step size $\stepsize$ specified in
Theorem~\ref{theorem:simple-convergence} and
Corollary~\ref{corollary:graph-catalog}. \\

\begin{figure}[t]
  \begin{center}
    \begin{tabular}{ccc}
      \hspace{-.5cm}
      \includegraphics[width=.34\columnwidth]{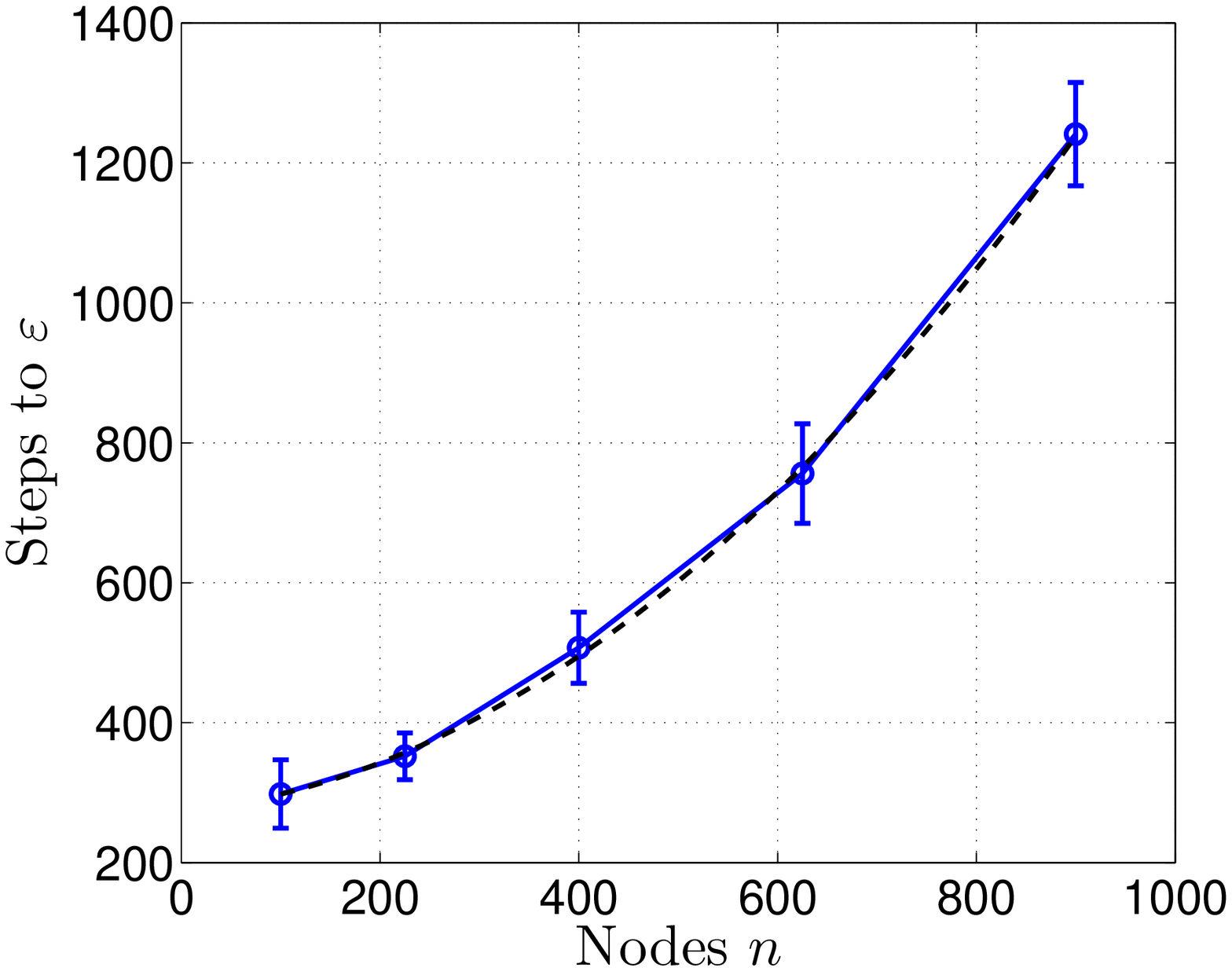} &
      \hspace{-.6cm}
      \includegraphics[width=.34\columnwidth]{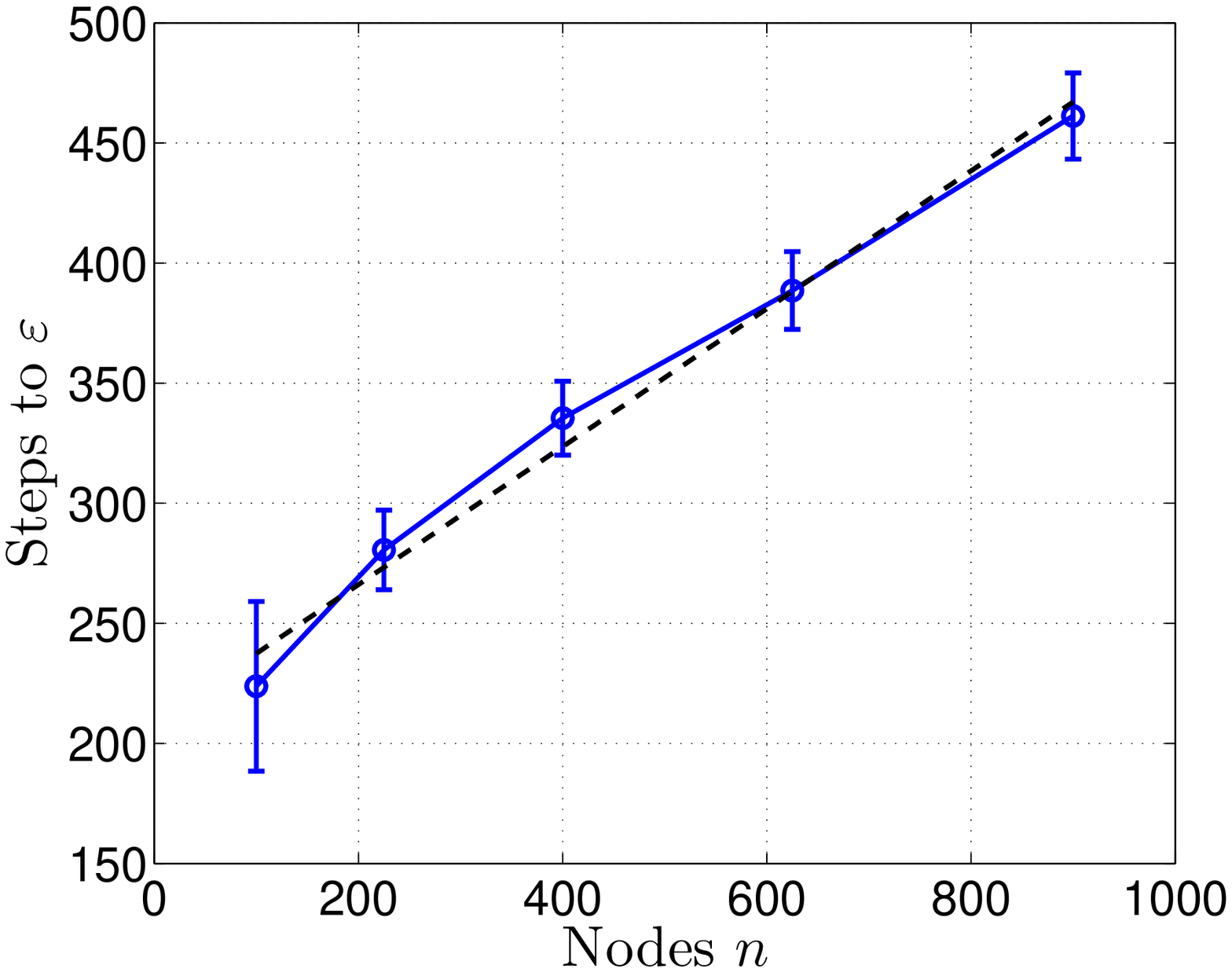} &
      \hspace{-.6cm}
      \includegraphics[width=.34\columnwidth]{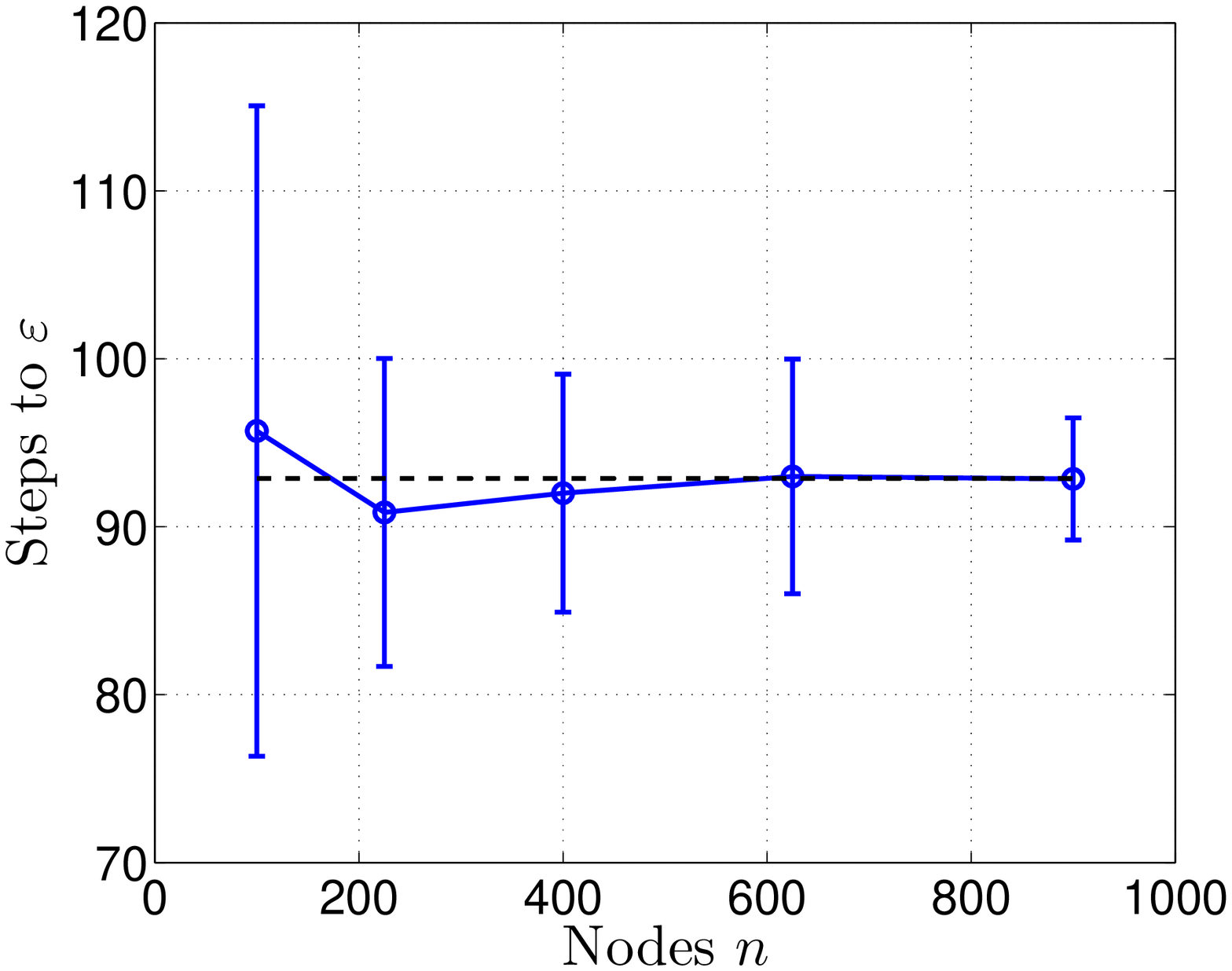} \\
      (a) & (b) & (c)
    \end{tabular}
    \caption{\label{fig:spectral-scale} Each plot shows the number of
      iterations required to reach a fixed accuracy $\epsilon$
      (vertical axis) versus the network size $n$ (horizontal axis).
      Each panel shows the same plot for a different graph topology:
      (a) single cycle; (b) two-dimensional grid; and (c) bounded
      degree expander.  Step sizes were chosen
      according to the spectral gap, and dotted lines show predictions
    of Corollary~\ref{corollary:graph-catalog}.}
  \end{center}
\end{figure}

Figure~\ref{fig:convergence} provides plots of the function error
$\max_i [f(\what{x}_i(T) - f(x^*)]$ versus the number of iterations
for grid graphs with a varying number of nodes $n \in \{225, 400, 625
\}$.  In addition to demonstrating
convergence, these plots also show how the convergence time scales as
a function of the graph size $n$.  In particular, for a given class of
optimization problems, define $\Tspec{\epsilon}{\numnode}$ to be the
number of iterations required to obtain $\epsilon$-accurate solution
for a graph $\graph$ with $n$ nodes.  As shown in
Figure~\ref{fig:convergence}, for any fixed $\epsilon > 0$, the
function $\Tspec{\epsilon}{n}$ shifts to the right as $n$ is
increased, and the goal of network scaling analysis is to gain a
precise understanding of this shifting.

As discussed following Corollary~\ref{corollary:graph-catalog}, for
cycles, grids, and expanders, we have the following upper bounds on
the quantity $\Tspec{\epsilon}{\numnode}$:
\begin{equation}
\label{EqnTheoretical}
  T_{\rm cycle}(\epsilon;\numnode) =
  \order\left(\frac{\numnode^2}{\epsilon^2}\right), ~~~ T_{\rm
  grid}(\epsilon; \numnode) =
  \order\left(\frac{\numnode}{\epsilon^2}\right), ~~~ {\rm and} ~~~
  T_{\rm expander}(\epsilon; \numnode) =
  \order\left(\frac{1}{\epsilon^2}\right).
\end{equation}
In Figure~\ref{fig:spectral-scale}, we compare these theoretical
predictions with the actual behavior of dual subgradient averaging.
Each panel shows the function $\Tspec{\epsilon}{\numnode}$ versus the
graph size $\numnode$ for the fixed value $\epsilon = 0.1$; the three
different panels correspond to different graph types: cycles (a),
grids (b) and expanders (c).  In each panel, each point on the blue
curve is the average of $\NUMSIM$ trials, and the bars show standard
errors. For
comparison, the dotted black line shows the theoretical
prediction~\eqref{EqnTheoretical}.  Note that the agreement between the
empirical behavior and theoretical predictions is excellent in all
cases.  In particular, panel (a) exhibits the quadratic scaling
predicted for the cycle, panel (b) exhibits the the linear scaling
expected for the grid, and panel (c) shows that expander graphs have
the desirable property of having constant network scaling. \\

%% \begin{figure}[t]
%%   \begin{center}
%%     \begin{tabular}{ccc}
%%       \hspace{-.5cm}
%%       \includegraphics[width=.34\columnwidth]{Images/confirm_cycle_scale}
%%       &
%%       \hspace{-.6cm}
%%       \includegraphics[width=.34\columnwidth]{Images/confirm_grid_scale} &
%%       \hspace{-.6cm}
%%       \includegraphics[width=.34\columnwidth]{Images/confirm_expander_scale} \\
%%       (a) & (b) & (c)
%%     \end{tabular}
%%     \caption{\label{fig:hard-graphs} Each plot shows the number of iterations
%%       required until each entry of the dual vector $z(t)$ becomes negative
%%       (vertical axis) versus network size $\numnode$ (horizontal axis) for a
%%       linear problem with optimal parameter $x^* = -1$. Each panel shows the
%%       same plot for a different graph topology: (left) singly-connected path,
%%       (middle) two-dimensional grid, and (right) expander.}
%%   \end{center}
%% \end{figure}

Though our focus in this paper is mostly a theoretical one, in our
final set of experiments we compare the distributed dual averaging
method (DDA) that we present to the Markov incremental gradient
descent (MIGD) method~\cite{JohanssonRaJo09} and the distributed
projected gradient method~\cite{RamNeVe10}, which seem to have the
sharpest convergence rates currently in the literature. In
Figure~\ref{fig:dda-vs-migd}, we plot the quantity
$\Tspec{\epsilon}{\numnode}$ versus graph size $\numnode$ for DDA and
MIGD on grid and expander graphs. We use the optimal stepsize
$\stepsize(t)$ suggested by the analyses for each method. (We do not
plot results for the distributed projected gradient
method~\cite{RamNeVe10} because the optimal choice of stepsize
according to the analysis therein results in such slow convergence
that it does not fit on the plots.) Fig.~\ref{fig:dda-vs-migd} makes
it clear that---especially on graphs with good connectivity properties
such as the expander in Fig.~\ref{fig:dda-vs-migd}(b)---the dual
averaging algorithm gives improved performance.

\begin{figure}[t]
  \begin{center}
    \begin{tabular}{cc}
      \includegraphics[width=.45\columnwidth]{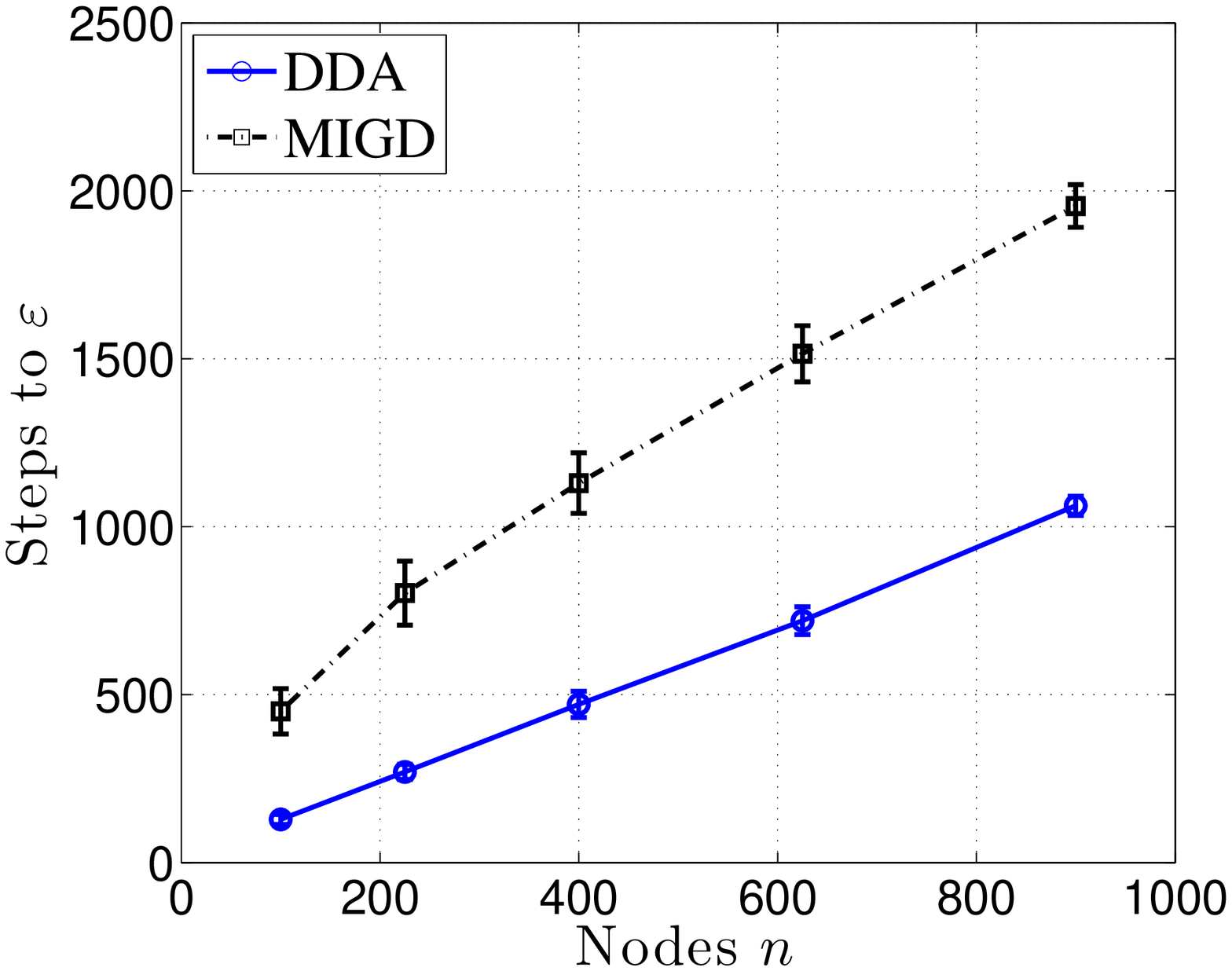} &
      \includegraphics[width=.45\columnwidth]{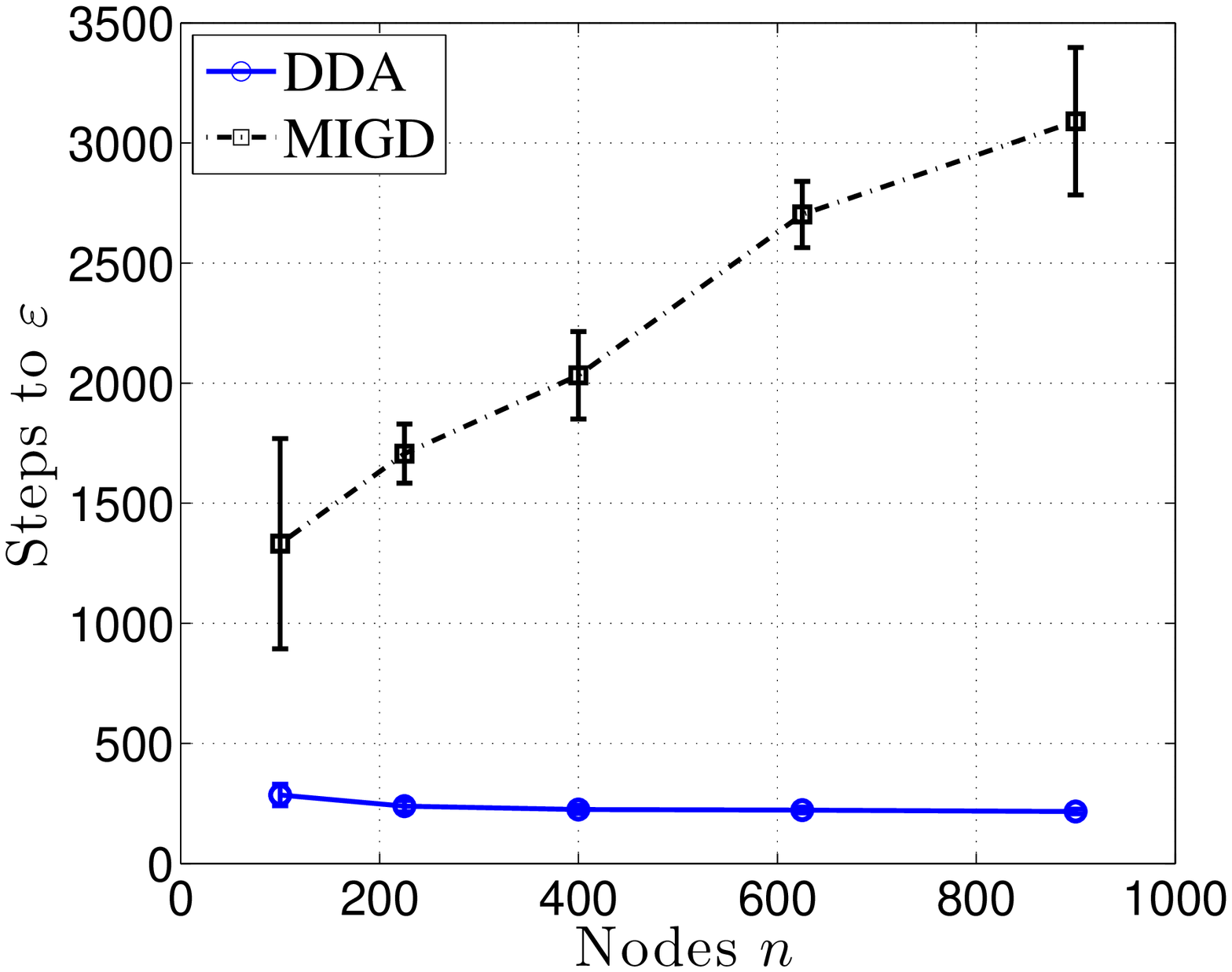} \\
      (a) & (b)
    \end{tabular}
    \caption{ \label{fig:dda-vs-migd} Each plot shows the number of
      iterations required to reach a fixed accuracy $\epsilon$
      (vertical axis) versus network size $\numnode$ (horizontal axis)
      for distributed dual averaging (DDA) and Markov incremental
      gradient descent (MIGD)~\cite{JohanssonRaJo09}. The panels show the
      same plot for different graph topologies: (a) two-dimensional grid,
      (b) bounded degree expander.}
  \end{center}
  \vspace{-.5cm}
\end{figure}

\section{Conclusions and Discussion}

In this paper, we proposed and analyzed a distributed dual averaging
algorithm for minimizing the sum of local convex functions over a
network. It is computationally efficient, and we provided a sharp
analysis of its convergence behavior as a function of the properties
of the optimization functions and the underlying network topology.
Our analysis demonstrates a close connection between convergence rates
and mixing times of random walks on the underlying graph; such a
connection is natural given the local and graph-constrained nature of
our updates.  In addition to analysis of deterministic updates, our
results also include the case of stochastic communication protocols,
for instance when communication occurs only along a random subset of
the edges at each round. Such extensions allow for the design of
protocols that provide interesting tradeoffs between the amount of
communication and convergence rates. We also demonstrate that our
algorithm is robust to noise by providing an analysis for the case of
stochastic optimization with noisy gradients.  We confirmed the
sharpness of our theoretical predictions by implementation and
simulation of our algorithm.

There are several interesting open questions that remain to be
explored. For instance, it would be interesting to analyze the
convergence properties of other kinds of network-based optimization
problems, by combining local information in different structures. It
would also be of interest to study what other optimization procedures
from the standard setting can be converted into efficient distributed
algorithms to better exploit problem structure when possible.

\subsection*{Acknowledgments}

AA was supported by a Microsoft Research Fellowship and JCD was
supported by the National Defense Science and Engineering Graduate
Fellowship (NDSEG) Program.  In addition, MJW was partially supported
by NSF-CAREER-0545862 and AFOSR-09NL184. We thank John Tsitsiklis for
a careful reading of the paper and providing helpful comments and
several anonymous reviewers for useful feedback.

\appendix

\section{The Dual Averaging Algorithm}
\label{app:lazy}

In this section, we give a simple convergence proof for the basic
(non-distributed) dual averaging algorithm~\eqref{eqn:standard-lazy}.
In particular, we recall the updates
\begin{equation*}
  z(t + 1) = z(t) + g(t) ~~~ {\rm and} ~~~ x(t + 1) =
  \argmin_{x \in \xdomain}
  \Big\{ \<z(t + 1), x\> + \frac{1}{\stepsize(t)} \prox(x)
  \Big\}.
\end{equation*}
Recall our assumptions without any loss of generality that $0 \in \xdomain$,
$\prox \ge 0$ and $\prox(0) = 0$.  Let $\bindic{x \in \xdomain}$ be the
$\{0,\infty\}$-valued indicator function for membership in $\xdomain$, and for
each $\stepsize > 0$, let $\proxdual_\stepsize$ denote the conjugate dual of
the convex function $\frac{1}{\stepsize}\prox(x) + \bindic{x \in \xdomain}$.
By definition, the conjugate takes the form
\begin{equation}
  \label{eqn:prox-dual}
  \proxdual_{\stepsize}(z) = \sup_{x \in \xdomain} \Big\{\<z, x\> -
  \frac{1}{\stepsize}\prox(x) \Big\}
  = \sup_x \Big\{ \<z, x\> -
  \frac{1}{\stepsize}\prox(x) - \bindic{x \in \xdomain} \Big\}.
\end{equation}
The definition~\eqref{eqn:projection} of the projection $\Pi_\xdomain^\prox$
shows that the supremum~\eqref{eqn:prox-dual} is uniquely attained by
$\Pi_\xdomain^\prox(-z, \stepsize)$. Moreover, for any fixed $z$, we note that
at $x = 0$, $\<z, x\> - \frac{1}{\stepsize}\prox(x) - \bindic{x \in \xdomain}
= 0$. Thus, we can restrict the supremum in \eqref{eqn:prox-dual} to the set
\[
\Big\{x ~ \mid ~ \frac{1}{\stepsize} \prox(x) + \bindic{x \in \xdomain}
- \<z, x\> \le 0 \Big\}
= \xdomain \cap \Big\{x ~ \mid ~ \frac{1}{\stepsize} \prox(x)
- \<z, x\> \le 0 \Big\},
\]
which is compact since $\xdomain$ is closed and $\prox$ is strongly convex.
Thus, since the supremum is uniquely attained and $\<z, x\>$ is differentiable
in $z$, $\nabla \proxdual_\stepsize(-z) = \Pi_\xdomain^\prox(z,
\stepsize)$~\cite[Theorem 4.4.2]{HiriartUrrutyLe96}.

This fact has two important consequences. First, since the projection
is Lipschitz-continuous (see Lemma~\ref{lemma:projection}), we have
the bound
\begin{equation*}
  \norm{\nabla \proxdual_\stepsize(-z) - \nabla
  \proxdual_\stepsize(-z - g)}
  = \norm{\Pi_\xdomain^\prox(z, \stepsize) -\Pi_\xdomain^\prox(z + g,
  \stepsize)}
  \leq \stepsize \dnorm{g}.
\end{equation*}
Consequently, an integration argument (e.g.,~\cite[Lemma 1.2.3]{Nesterov04})
yields the upper bound
\begin{equation}
  \label{eqn:lipschitz-dual}
  \proxdual_\stepsize(-z - g) \le \proxdual_\stepsize(-z) - \<g,
  \nabla\proxdual_\stepsize(-z)\> + \half \stepsize \dnorm{g}^2.
\end{equation}
The second consequence is that we have
\begin{align*}
  x(t) & = \nabla\proxdual_{\stepsize(t-1)}(-z(t)) =
  \Pi_{\xdomain}(z(t), \stepsize(t-1)).
\end{align*}

%%%%%%%%%%%%%%%%%%%%%%%%%%%%%%%%%%%%%%%%%%%%%%%%%%%%%%%%%%%%%%%%%%

\subsection{Proof of Lemma~\ref{lemma:ftrl-linear}}
\label{AppLemftrl-linear}

To bound the sequence of inner products, we note that for any $x^* \in
\xdomain$, we have
\begin{align}
  -\sum_{t = 1}^T \<g(t), x^*\> & \le \sup_{x \in
    \xdomain}\left\{-\sum_{t=1}^T \<g(t), x\> -
  \frac{1}{\stepsize(T)}\prox(x)\right\} + \frac{1}{\stepsize(T)}
  \prox(x^*) \nonumber \\ & = \proxdual_{\stepsize(T)}(-z(T+1)) +
  \frac{1}{\stepsize(T)}\prox(x^*).
  \label{eqn:bound-x-star}
\end{align}

By definition of the conjugate function $\proxdual_\stepsize$, whenever we
have $\stepsize(t) \le \stepsize(t - 1)$, then we are guaranteed that
$\proxdual_{\stepsize(t)}(z) \le \proxdual_{\stepsize(t - 1)}(z)$ for all $z
\in \R^d$. Thus, using the upper bound~\eqref{eqn:lipschitz-dual} and the
relations $x(t) = \nabla\proxdual_{\stepsize(t - 1)}(-z(t))$ and $z(t + 1) =
z(t) + g(t)$, we obtain
\begin{align*}
  \proxdual_{\stepsize(t)}(-z(t + 1)) & \leq  \proxdual_{\stepsize(t -
  1)}(-z(t + 1))  \\
& = \proxdual_{\stepsize(t - 1)}(-z(t) - g(t)) \\
& \leq \proxdual_{\stepsize(t - 1)}(-z(t)) - \<g(t), x(t)\> + \half
  \stepsize(t - 1) \dnorm{g(t)}^2.
\end{align*}
Rearranging terms yields
\begin{equation}
  \<g(t), x(t)\> \leq \proxdual_{\stepsize(t - 1)}(-z(t)) -
  \proxdual_{\stepsize(t)}(-z(t+1)) + \half \stepsize(t -
  1)\dnorm{g(t)}^2.
  \label{eqn:inductive-step}
\end{equation}

Finally, we combine the upper bound on $\<g(t), x(t)\>$ from equation
\eqref{eqn:inductive-step} with the earlier
bound~\eqref{eqn:bound-x-star}, thereby obtaining that for any $x^*
\in \xdomain$, the sum $S(T) = \sum_{t=1}^T \<g(t), x(t) - x^*\>$ is
upper bounded as
\begin{align*}
  \lefteqn{S(T) \leq \sum_{t=1}^T \<g(t), x(t)\> +
  \proxdual_{\stepsize(T)}(-z(T+1)) + \frac{1}{\stepsize(T)}
  \prox(x^*)} \\
  & \leq \half \sum_{t=1}^T \stepsize(t - 1)\dnorm{g(t)}^2 +
  \sum_{t=1}^T \big[\proxdual_{\stepsize(t - 1)}(-z(t)) -
    \proxdual_{\stepsize(t)}(-z(t + 1)) \big] +
  \proxdual_{\stepsize(T)}(-z(T + 1)) + \frac{1}{\stepsize(T)}
  \prox(x^*) \quad \\
  & = \half \sum_{t=1}^T \stepsize(t - 1)\dnorm{g(t)}^2 +
  \frac{1}{\stepsize(T)} \prox(x^*).
\end{align*}
The last line exploited the facts that $z(1) = 0$ and $\proxdual_\stepsize(0)
= 0$.  This completes the proof of the claim.

%%   We prove the statement by induction. For the base case, choose $x(1) =
%%   \argmin_{x \in \xdomain} \prox(x)$. Then, since $x(2)$ is the minimizer of
%%   $\stepsize \<g(1), x\> + \prox(x)$ in $\xdomain$, for $x^* \in
%%   \xdomain$ we have
%%   \begin{equation*}
%%   \stepsize\<g(1), x(1) - x^*\> -  \prox(x^*)
%%   \leq \stepsize \<g(1), x(1) - x(2)\> - \prox(x(2))
%%   \le \stepsize \<g(1), x(1) - x(2)\> - \prox(x(1)).
%%   \end{equation*}
%%   Thus we have our inductive hypothesis, that is, that
%%   \begin{equation}
%%     \label{eqn:l2-ftrl-induction}
%%     \stepsize\sum_{t=1}^T \<g(t), x(t) - x^*\> - \prox(x^*)
%%     \le \stepsize \sum_{t=1}^T \<g(t), x(t) - x(t+1)\> - \prox(x(1)).
%%   \end{equation}
%%   Now consider $\stepsize \sum_{t=1}^{T+1} \<g(t), x\> + 
%%   \prox(x)$. Since $x(T+2)$ minimizes the sum over $\xdomain$, we have
%%   \begin{align*}
%%     \lefteqn{
%%       \stepsize \sum_{t=1}^{T + 1}\<g(t), x(t) - x^*\> - \prox(x^*)
%%       \le \stepsize \sum_{t=1}^{T+1}\<g(t), x(t) - x(T + 2)\>
%%       - \prox(x(T+2))} \\
%%     & = \stepsize \sum_{t=1}^T \<g(t), x(t) - x(T + 2)\>
%%     - \prox(x(T+2))
%%     + \stepsize \<g(T + 1), x(T + 1) - x(T+2)\> \\
%%     & \le \stepsize \sum_{t=1}^T \<g(t), x(t) - x(t+1)\> - \prox(x(1))
%%     + \stepsize \<g(T+1), x(t+1) - x(T+2)\> \\
%%     & = \stepsize \sum_{t=1}^{T+1} \<g(t), x(t) - x(t + 1)\> -
%%     \prox(x(1)).
%%   \end{align*}
%%   The third line follows from the inductive hypothesis in
%%   \eqref{eqn:l2-ftrl-induction} applied with $x^* = x(T+2)$.

%%%%%%%%%%%%%%%%%%%%%%%%%%%%%%%%%%%%%%%%%%%%%%%%%%%%%%%%%%%%%%%%%%%%%%%%%%%

\subsection{Proof of Lemma~\ref{lemma:transfer-y-to-x}}
\label{Applemma:transfer-y-to-x}

Via the $L$-Lipschitz continuity of the $f_i$, we can write
\begin{align*}
  \sum_{t=1}^T f(x_i(t)) - f(x^*)
  & = \sum_{t=1}^T f(y(t)) - f(x^*) + \sum_{t=1}^T f(x_i(t)) - f(y(t)) \\
  & \le \sum_{t=1}^T f(y(t)) - f(x^*) + \sum_{t=1}^T  L \norm{x_i(t)
    - y(t)}.
\end{align*}
For the second bound, we again use the $L$-Lipschitz continuity of the $f_i$
and the triangle inequality,
\begin{align*}
  f(\what{x}_i(T)) - f(x^*)
  &= f(\what{y}(T)) - f(x^*) + f(\what{x}_i(T)) - f(\what{y}(T)) \\
  & \le f(\what{y}(T)) - f(x^*) + L \norm{\what{x}_i(T) - \what{y}(T)}
  \le f(\what{y}(T)) - f(x^*)
  + \frac{L}{T} \sum_{t=1}^T\norm{x_i(t) - y(t)}.
\end{align*}
Lipschitz-continuity of the projection (Lemma~\ref{lemma:projection}) shows
that $\norm{x_i(t) - y(t)} \le \stepsize(t)\dnorm{\bar{z}(t) - z_i(t)}$ which
gives both the desired results.

\subsection{Lipschitz continuity of projections}
\label{Applemma:projection}
The following lemma on the Lipschitz-continuity of the projection
operator is well-known, but we state and prove it for completeness.
\begin{lemma}
  \label{lemma:projection}
  For an arbitrary pair $u, v \in \R^d$, we have
  \begin{equation*}
    \norm{\Pi_\xdomain^\prox(u, \stepsize) - \Pi_\xdomain^\prox(v, \stepsize)}
    \le \stepsize\dnorm{u - v}.
  \end{equation*}
\end{lemma}
\begin{proof}
Lemma~\ref{lemma:projection} is essentially an immediate consequence
of the relationship between strong-convexity and Lipschitz continuity
of the gradient for conjugate functions~\cite[Theorem
X.4.2.1]{HiriartUrrutyLe96b}, but we give a short proof for
completeness.  For an arbitrary pair $u, v \in \R^d$, denote \mbox{$w
= \Pi_\xdomain^\prox(u, \alpha)$} and \mbox{$x = \Pi_\xdomain^\prox(v,
\alpha)$.}  By the first-order optimality conditions for convex
minimization, for any $y \in \xdomain$, we have
\begin{equation*}
  \<u + \frac{1}{\stepsize}\nabla\prox(w), y - w\> \ge 0
  ~~~ {\rm and} ~~~
  \<v + \frac{1}{\stepsize}\nabla\prox(x), y - x\> \ge 0.
\end{equation*}
Setting $y = x$ and $y = w$ in these two inequalities (respectively)
yields
\begin{equation*}
  \<\stepsize u + \nabla\prox(w), x - w\> \ge 0 ~~~ {\rm and} ~~~
  \<\stepsize v + \nabla\prox(x), w - x\> \ge 0.
  \end{equation*}
Adding the above two inequalities, we obtain the bound
  \begin{equation}
    \<\nabla \prox(w) - \nabla\prox(x), w - x\> \le \stepsize \<u - v,
    x - w\> \le \stepsize\dnorm{v_1 - v_2}\norm{w - x}.
    \label{eqn:gradnormlower}
  \end{equation}
  On the other hand the strong convexity of $\psi$ implies that
  $\psi(w) \ge \psi(x) + \<\psi(x), w - x\> + \half \norm{x - w}^2$,
  with an analogous bound with the roles of $x$ and $w$ exchanged.
  Some algebra then leads to
  \begin{equation*}
  \<\nabla\prox(w) - \nabla\prox(x), w - x \> \ge \norm{w - x}^2,
  \end{equation*}
  which, when combined with~\eqref{eqn:gradnormlower}, gives
  the desired result.
\end{proof}

\section{Background on stochastic matrices}
\label{app:markovchain}

In this section, we briefly review some well-known properties of
stochastic matrices; we refer the reader to Chapter 8 of Horn and
Johnson~\cite{HornJo85} for additional detail. For an $n \times n$
matrix $A$, we let its singular values be $\sigma_1(A) \geq
\sigma_2(A) \geq \cdots \ge \sigma_n(A)$, and for a real symmetric
$A$, we define the eigenvalues $\lambda_1(A) \geq \lambda_2(A) \geq
\cdots \geq \lambda_n(A)$. Let $\onevec$ be the all ones vector. In
our setting, $\stochmat = [\stochvec_1 ~ \cdots ~ \stochvec_n] \in
\R^{n \times n}$ is a doubly stochastic matrix, so that
$\stochmat\onevec = \onevec$ and $\onevec^T \stochmat =
\onevec^T$. We have $\sigma_1(\stochmat) = 1$,
$\lambda_1(\stochmat^T\stochmat) = 1$, and $1-\sigma_2(\stochmat)$ is
the spectral gap, which is known to determine the mixing properties of
the Markov chain induced by $\stochmat$~\cite{LevinPeWi08}.

In order to establish the connection between mixing and spectral gap,
define the uniform matrix \mbox{$\Fmat \defeq \frac{1}{n}\onevec \onevec^T$}.
Observe that $\Fmat$ is idempotent ($\Fmat^2 = \Fmat$), and moreover
it satisfies $\stochmat \Fmat = \Fmat \stochmat = \Fmat$.  By
construction, the eigenspectrum of $\stochmat - \Fmat$ is equal
to that of $\stochmat$ except that the largest eigenvalue $1$ is
removed. Similarly, the eigenspectrum of $(\stochmat -
\Fmat)^T(\stochmat - \Fmat) = \stochmat^T \stochmat - \Fmat^T
\stochmat - \stochmat^T \Fmat + \Fmat^T \Fmat = \stochmat^T \stochmat -
F^T \Fmat$ is identical to that of $\stochmat^T\stochmat$ but with
$\lambda_1(\stochmat^T\stochmat) = 1$ removed.  Given these
properties, a simple calculation yields that for any integer $t =
1,2,\ldots$, we have $(\stochmat - \Fmat)^t = \stochmat^t - \Fmat$.
Consequently, for any $x \in \R^n$, we have
\begin{equation*}
  \ltwo{\stochmat^t x - \Fmat x} = \ltwo{(\stochmat - \Fmat)^t x} \le
  \sigma_1(\stochmat - \Fmat)\ltwo{(\stochmat - \Fmat)^{t-1}x} \le
  ~ \cdots ~ \le \sigma_2(\stochmat)^t \ltwo{x}.
\end{equation*}
If we take $x = e_i$, denoting a canonical basis vector for
$i=1,\dots,n$, then we see that $\linf{\stochmat^t - L} \le
\sigma_2(\stochmat)^t$. Taking $x \in \Delta_n$, the $n$-dimensional
simplex, gives 
\begin{equation*}
\tvnorm{\stochmat^tx - \onevec/n}
= \half \lone{\stochmat^tx - \onevec/n}
\le \half \sqrt{n}\ltwo{\stochmat^tx - \onevec/n}
\le \half \sigma_2(\stochmat)^t \sqrt{n},
\end{equation*}
which establishes the bound~\eqref{eqn:tv-bound}.  (The $\sqrt{n}$ factor in
the bound is standard in the Markov chain literature,
e.g.,~\cite[Proposition 3]{DiaconisSt91}.)

\section{Eigenvalues of paths}
\label{sec:k-connected-path}

Let $G$ be a graph and $S$ be a subset of the nodes in the graph. Let
$E(S, S^c)$ denote the set of edges crossing between $S$ and $S^c$,
and let the volume of $S$ be the sum of the degrees of the nodes in
$S$, that is, $\vol(S) = \sum_{i \in S} \degree_i$. The Cheeger
constant of a graph $G$ is defined as
\begin{equation}
  \label{eqn:cheeger-constant}
  h_G \defeq \min_{S \subset \vertex} \frac{\card(E(S,
  S^c))}{\min\{\vol(S), \vol(S^c)\}}.
\end{equation}
If $\laplacian$ is the Laplacian of $G$, then $2 h_G \ge
\lambda_{n-1}(\laplacian) > \half h_G^2$ (e.g., see Lemma 2.1 and
Theorem 2.2 in Chung~\cite{Chung98}).

\begin{lemma}
  \label{proposition:k-connected-path}
  Let $G$ be a $k$-connected path with $n$ nodes and $k \le \sqrt{n}$.
  Then its normalized graph Laplacian $\laplacian$ satisfies
  $\lambda_{n-1}(\laplacian) = \Theta(k^2/n^2)$.
\end{lemma}
\begin{proof}
We invoke Theorem 4.13 in Chung~\cite{Chung98} to conclude that
$\lambda_{n-1}(\laplacian) = \order(k^2/n^2)$, since $G$ is a subgraph of the
$k$-connected cycle. It thus suffices to prove that the Cheeger constant is
lower bounded as $h_\graph = \Omega(k/n)$.

Let $\Sstar$ be the set of nodes achieving the minimum in the
definition~\eqref{eqn:cheeger-constant}. To make the rest of the proof
easier, assume that the degree of each node is $2k$.  (We may do so
without loss of generality, since it only has the effect of increasing
$\vol(\Sstar)$ and $\vol(\Sstar^c)$ in the Cheeger constant
calculation, and so any Cheeger constant calculated under this
assumption lower bounds the true Cheeger constant.)

First, note that one of the nodes in $\Sstar$ must be against the end
of the path---if not, shifting the nodes in $\Sstar$ in one direction
(taking into account that we must pick the direction in which more
nodes are brought near the end of the path) can only decrease
$\card(E(\Sstar, \Sstar^c))$. Now we show that all of the nodes in
$\Sstar$ must be directly adjacent to one another. Suppose the nodes
are not adjacent. Since $k \le \sqrt{n}$, there must be a pair of
nodes in $\Sstar$ with a distance of at least $k$. Let $i \in
\Sstar^c$ be between those two nodes, and let $\Sstar_\ell$ denote the
nodes to the left of $i$ and $\Sstar_r$ the nodes to the
right. Collapsing all the nodes in $\Sstar_r$ to the rightmost end of
the path and all the nodes in $\Sstar_\ell$ to the leftmost end can
only decrease $\card(E(\Sstar, \Sstar^c))$. If $|\Sstar| \ge k$, then
at least one of the sets $\Sstar_r$ and $\Sstar_\ell$ shares $k(k -
1)/4$ edges with $\Sstar^c$.  Otherwise, if $|\Sstar| < k$, then
$\card(E(\Sstar, \Sstar^c)) \ge k$ and $\vol(\Sstar) \le k^2$, so
$\card(E(\Sstar, \Sstar^c)) / \vol(S) \ge 1/k$.  Under the assumption
$k^2 \leq n$, we have $1/k \leq k/n$, from which the result follows.
\end{proof}

%%%%%%%%%%%%%%%%%%%%%%%%%%%%%%%%%%%%%%%%%%%%%%%%%%%%%%%%%%%%%%%%%%%%%%%%%%%%

\section{Composite Objectives}
\label{app:composite}

In this section, we show how to generalize the dual averaging
algorithm to incorporate composite objectives, specifically those of
the form $f + \varphi$ for known $\varphi$. Though it is possible to
perform similar derivations to those in Lemma~\ref{lemma:ftrl-linear},
for brevity we refer to recent work of
Xiao~\cite{Xiao10}. Nonetheless, the algorithm is conceptually very
similar to the dual averaging algorithm
(updates~\eqref{eqn:lazy-unproject} and~\eqref{eqn:lazy-project}), and
equally as simple to write.  We assume that $\varphi$ is closed convex
and non-negative, and $\xdomain$ is closed. We define the composite
projection operator $\compositeproject{t}$ as
\begin{equation}
  \label{eqn:composite-project}
  \compositeproject{t}(z) = \argmin_{x \in \xdomain} \Big\{ \<z, x\>
  + t \varphi(x) + \frac{1}{\stepsize(t)} \psi(x)\Big\}.
\end{equation}
The mapping $\compositeproject{t}$ is $\stepsize(t)$-Lipschitz with respect to
$\norm{\cdot}$ and $\dnorm{\cdot}$, that is,
\begin{equation}
  \label{eqn:composite-lipschitz}
  \norm{\compositeproject{t}(z_1) - \compositeproject{t}(z_2)}
  \le \stepsize(t) \dnorm{z_1 - z_2}.
\end{equation}
As in Lemma~\ref{lemma:projection}, \eqref{eqn:composite-lipschitz} is
a consequence of the fact that the conjugate dual of a
$1/\stepsize(t)$-strongly convex function has $\stepsize(t)$-Lipschitz
continuous gradient with respect to the associated dual norm, and the
gradient of the conjugate of $t \varphi(x) + \frac{1}{\stepsize(t)}
\prox(x)$ is simply $\compositeproject{t}(z)$~\cite[Theorem
  X.4.2.1]{HiriartUrrutyLe96b}.

The distributed algorithm based on the
update~\eqref{eqn:composite-project} is essentially identical to the
dual averaging algorithm discussed in the main body of the paper. Each
agent $i$ maintains the gradient vector
\begin{equation}
  z_i(t + 1) = \sum_{j=1}^n \stochvec_{ij}(t) z_j(t) - g_i(t)
  ~~~ {\rm where} ~~~
  \E g_i(t) \in \partial f_i(x_i(t)).
  \label{eqn:composite-lazy-unproject}
\end{equation}
The update to $x_i(t + 1)$ is then
\begin{equation}
  \label{eqn:composite-lazy-project}
  x_i(t + 1) = \compositeproject{t}(-z_i(t + 1)).
\end{equation}
As in \eqref{eqn:mean-z-update}, we have $\bar{z}(t + 1) = \bar{z}(t) -
\frac{1}{n} \sum_{i=1}^n g_j(t)$. The following proposition, a simplification
of~\cite[Section B.2]{Xiao10}, allows us to give a convergence guarantee for
the algorithm described by \eqref{eqn:composite-lazy-unproject} and
\eqref{eqn:composite-lazy-project}.
\begin{proposition}
  \label{proposition:linear-rda}
  Let $\stepsize(t)$ be a decreasing sequence and $g(t) \in \R^d$ be an
  arbitrary sequence of vectors. If $x(t + 1) =
  \compositeproject{t}(\sum_{\tau=1}^t g(t))$, then for any $x^* \in \xdomain$,
  \begin{equation*}
  \sum_{t=1}^T \<g(t), x(t) - x^*\> + \varphi(x(t)) - \varphi(x^*)
  \le \frac{1}{\stepsize(T)}\psi(x^*)
  + \half \sum_{t=1}^T \stepsize(t-1)\norm{g(t)}_*^2.
  \end{equation*}
\end{proposition}

The above proposition, combined with the techniques used to derive
Theorem~\ref{theorem:master-convergence}, allow us to easily prove convergence
of distributed composite-objective dual averaging. As earlier, let $y(t) =
\compositeproject{t}(-\bar{z}(t))$, and assume that the $f_i$ are
$L$-Lipschitz with respect to $\norm{\cdot}$. Then as in
\eqref{eqn:dist-lazy-to-bound}, \eqref{eqn:first-order-dist-bound}, and
\eqref{eqn:first-order-y-replacement}, for any $x^* \in \xdomain$, we
immediately have
\begin{align*}
  \lefteqn{\sum_{t=1}^T \big[
      f(y(t)) + \varphi(y(t)) - f(x^*) - \varphi(x^*)\big]} \\
  & \le \sum_{t=1}^T \frac{1}{n}\<\sum_{i=1}^n g_i(t), y(t) - x^*\>
  + \sum_{t=1}^T \varphi(y(t)) - \varphi(x^*)
  + \sum_{t=1}^T \sum_{i=1}^n \frac{2L}{n}\norm{y(t) - x_i(t)}.
\end{align*}
By definition of $y(t)$, we see that Proposition~\ref{proposition:linear-rda}
bounds the above by
\begin{equation*}
\frac{1}{\stepsize(T)} \psi(x^*) + \half \sum_{t=1}^T
\stepsize(t - 1) L^2 + \sum_{t=1}^T \sum_{i=1}^n \frac{2L}{n} \norm{y(t)
  - x_i(t)}.
\end{equation*}
Finally, we use the fact that the mapping $\compositeproject{t}$ is
$\stepsize(t)$-Lipschitz to see that for the distributed composite-objective
projection algorithm of \eqref{eqn:composite-lazy-unproject} and
\eqref{eqn:composite-lazy-project},
\begin{equation}
  \label{eqn:composite-lazy-convergence}
  \sum_{t=1}^T f(y(t)) + \varphi(y(t)) - f(x^*) - \varphi(x^*)
  \le \frac{1}{\stepsize(T)} \psi(x^*)
  + \half \sum_{t=1}^T \stepsize(t - 1) L^2
  + \frac{2 L}{n} \sum_{t=1}^T \stepsize(t) \sum_{i=1}^n
  \norm{\bar{z}(t) - z_i(t)}_*.
\end{equation}
Any of the techniques in the prequel can be used to bound
\eqref{eqn:composite-lazy-convergence}.

%% Indeed, using an upper bound (say
%% $LZ$) on $\dnorm{\bar{z}(t) - z_i(t)}$ as in the proof of
%% Theorem~\ref{theorem:simple-convergence}, we can simplify
%% \eqref{eqn:composite-lazy-convergence}. Let $\stepsize(t) = c/\sqrt{t+1}$.
%% Note that $\sum_{t=1}^T \stepsize(t-1) \le 2c\sqrt{T}$. Then
%% \begin{equation*}
%% \sum_{t=1}^T f(y(t)) + \varphi(y(t)) - f(x^*) - \varphi(x^*)
%% \le \frac{\sqrt{T}}{c}\prox(x^*) + cL^2 \sqrt{T}
%% + 4c L^2 Z \sqrt{T}.
%% \end{equation*}
%% Assuming that $\prox(x^*) \le \radius^2$, dividing by $T$ and setting $c
%% \propto \radius / (L\sqrt{Z})$ gives
%% \begin{equation}
%%   f(\what{y}(T)) + \varphi(\what{y}(T)) - f(x^*) - \varphi(x^*)
%%   = O\left(\frac{\radius L \sqrt{Z}}{\sqrt{T}}\right).
%%   \label{eqn:decreasing-stepsizes}
%% \end{equation}
%% We see that the distributed composite-objective lazy projection algorithm has
%% essentially the same convergence guarantees as the simple distributed
%% lazy-projection algorithm.

\bibliographystyle{amsalpha}
\bibliography{bib}

\end{document}